%% file: SIIMS_Convergence_Analysis.tex
\documentclass[review,hidelinks,onefignum,onetabnum]{siamonline220329}
\nolinenumbers

\input{ex_shared}

\begin{document}
\maketitle
\begin{abstract}
Continuous conformal transformation minimizes the conformal energy. The convergence of minimizing discrete conformal energy when the discrete mesh size tends to zero is an open problem. This paper addresses this problem via a careful error analysis of the discrete conformal energy. Under a weak condition on triangulation, the discrete function minimizing the discrete conformal energy converges to the continuous conformal mapping as the mesh size tends to zero.   
\end{abstract}



\section{Introduction}

Given a 2D Riemann manifold $\mathcal{M}\subset{\cal R}^m$ with $m\geq 2$, an allowable mapping $f:\mathcal{M}\to {\cal R}^{n}$ is called to be conformal if $f$ preserves the angle of intersection of any two intersecting arcs on $\mathcal{M}$. Conformal transformation is commonly used in computer graphics and computer aided designs such as surface restriction \cite{XGYW04,MHWW17}, minimal surface generation \cite{Hutc91,UPKP93}, image morphing \cite{XFYF15,MHWW17}, and texture mapping \cite{SHSA00,LBPS02}. 

Conformal transformation can be verified via variant methodologies \cite{EK91book}. First, if $\mathcal{M}$ can be parameterized as $\mathbf{x} = \mathbf{x}(u_1,u_2)$ with variables $(u_1,u_2)\in \Omega \subset \mathbb{R}^2$, the conformality is equivalent to the first fundamental form equation of the two surfaces
\[
    \nabla \mathbf{y}^\top \nabla \mathbf{y} 
    = \eta\nabla\mathbf{x}^\top \nabla\mathbf{x}
\]
holds for the parameterization $\mathbf{x}$ of ${\cal M}$ and the transformed 
parameterization $\mathbf{y} = f\circ\mathbf{x}$ of ${\cal N} = f({\cal M})$,
where $\eta(u_1,u_2) >0$ is a scalar function, 
$\nabla \mathbf{x} = \big[\frac{\partial\mathbf{x}}{\partial u_1},\frac{\partial\mathbf{x}}{\partial u_2}\big]$ is the gradient matrix with respective to its variables, and so is $\nabla\mathbf{y}$.\footnote{We always assume that a vector function is in column form in this paper.} 
Second, the conformality can also be characterized by the minimization of conformal energy  
\[
    \mathcal{E}_C(f)=\frac{1}{2}\int_{\mathcal{M}} \|\nabla_{\!\cal M} f(\mathbf{x})\|^2_F \text{d}\sigma - {\cal A}(f),
\]
where $\nabla_{\!\cal M} f(\mathbf{x})= \nabla (f\circ{\mathbf x})
    \big((\nabla\mathbf{x}^\top \nabla\mathbf{x})^{-1}\nabla\mathbf{x}^\top\big)\in {\cal R}^{n\times m}$ 
is the tangential gradient of $f$ on $\cal M$ and ${\cal A}(f)$ is the area of $f({\cal M})$. It is known that the conformal energy $\mathcal{E}_C(f)$ is always nonnegative for any $f$, and $f$ is conformal if and only if $\mathcal{E}_C(f) = 0$.\footnote{A proof of $\mathcal{E}_D(f)\geq {\cal A}(f)$ is given in \cite{Hutc91} for the special case when both ${\cal M}$ and ${\cal N}$ are subsets of ${\cal R}^2$.} See Appendix \ref{Appendix:CE} for a simple proof of the property.
Third, a conformal mapping can be in a compounded form $f = g_2 \circ g_1$ of two quasi-conformal mappings $g_1: \mathcal{M}_1 \to \mathcal{M}_2$ and $g_2: \mathcal{M}_2 \to \mathcal{M}_3$ such that the Beltrami differential of $g_1^{-1}$ and $g_2$ are the same \cite[Theorem 1]{PTKC15}. 
Fourth, for a simple closed surface ${\cal M}$ and $p\in{\cal M}$, the conformal mapping from ${\cal M}\setminus\{p\}$ to the ${\cal R}^2$ satisfies the inhomogeneous Laplace-Beltrami equation \cite{SHSA00}
\begin{align}\label{Laplace-Beltrami}
    \Delta_\mathcal{M} f 
    = \Big(\frac{\partial}{\partial u}, -\frac{\partial}{\partial v} \Big) \delta_p, 
\end{align}
where $\delta_p$ is a Dirac delta function at $p$, {\it i.e.}, $\delta_p(p)=\infty$ and $\delta_p=0$ otherwise, and $u,v$ are the basis of tangent space of ${\cal M}$ at $p$. The solution is unique within a reflection because of the basis choice.

Theoretically, if the target surface ${\cal N}$ is predicted and a conformal mapping from $\cal M$ to ${\cal N}$ exists, the conformal mapping can be obtained via minimizing the conformal energy, 
\begin{align}\label{prob:CP}
    \min_{f: f({\cal M})={\cal N}}\mathcal{E}_C(f) 
    = \min_{f: f({\cal M})={\cal N}}\Big\{
    \frac{1}{2}\int_{\mathcal{M}} \|\nabla_{\!\cal M} f(\mathbf{x})\|_F^2 \text{d}\sigma -\mathcal{A}(f)\Big\}.
\end{align}
The quasi-conformal approach given in \cite{PTKC15}
chooses a function $g_1$ from ${\cal M}$ to ${\cal N}$ and determines the Beltrami differential of $g_1^{-1}$ first. Then it finds a transformation $g_2:{\cal N}\to{\cal N}$ by solving the Beltrami equation with the Beltrami differential of $g_1^{-1}$. The Laplace-Beltrami equation can be solved in its weak formulation, that is, for any smooth $g$,
\begin{align}\label{LB_weak}
    \int_{\cal M}\big\langle\nabla_{\cal M}f,\nabla_{\cal M}g\big\rangle d\sigma
    = -\int_{\cal M} \langle g,\Delta_{\cal M}f\rangle d\sigma
    = -\int_{\cal M} \big\langle g,\big(\frac{\partial}{\partial u}, 
        -\frac{\partial}{\partial v} \big) \delta_p\big\rangle d\sigma
    = \big\langle \big(\frac{\partial}{\partial u}, 
        -\frac{\partial}{\partial v} \big),g \big\rangle\big|_p.
\end{align}

Given a set $V$ of vertices $\{v_\ell\}$ sampled from $\cal M$, 
disctretization approaches of pursuing a conformal mapping proposed in the literature commonly approximate the smooth surface ${\cal M}$ by a triangle surface $S_V$ and focus the mappings on the piecewise linear functions $\{f^h\}$ of the triangle surface. 
For instance, the continuous Dirichlet energy ${\cal E}_D(f) = \frac{1}{2}\int_{\cal M}\|\nabla_{\cal M}f\|^2_Fd\sigma$ can be approximated by or discretized as
$\frac{1}{2}\int_{S_V}\|\nabla_{S_V}f^h\|^2_Fd\sigma$. In \cite{XGST20book}, each triangle is represented by barycentric coordinates, and the discrete Dirichlet energy has a quadratic form
\begin{align}\label{discrete DE}
    \frac{1}{2}\int_{S_V}\|\nabla_{S_V}f^h\|^2_Fd\sigma
    = \frac{1}{2}\langle L\mathbf{f}^h,\mathbf{f}^h\rangle,
\end{align}
where $L$ is a Laplacian matrix constructed by the cotangent angles of the triangles, and $\mathbf{f}^h$ is a matrix of the rows $\{f^h(v_\ell)^\top\}$.
The weak form (\ref{LB_weak}) can also be discretized as \cite{SHSA00}
\[
    \int_{S_V}\big\langle\nabla_{S_V}f^h,\nabla_{S_V}g^h\big\rangle \text{d}\sigma 
    = \langle \big(\frac{\partial}{\partial u}, 
        -\frac{\partial}{\partial v} \big),g^h \rangle\big|_p
\]
for any piecewise linear function $g^h$. 
The Beltrami differential for the quasi-conformal mappings $f_1$ and $f_2$ is discretized by directly applying the the Beltrami differential on a piecewise liner function $f^h$ on $S_V$. 
In Appendices \ref{Appendix:DE S_V} and \ref{Appendix:weakLB}, we derive (\ref{discrete DE}) without using the barycentric coordinate, and a discrete linear equation $L\mathbf{f}^h = \mathbf{b}$ equivalent to the discrete weak form, in a simple way. 

The algorithm CCEM given in \cite{YCWW21} computes an approximately discrete conformal mapping $f^h$ by solving 
\begin{align}\label{prog:DCE_S}
    \min_{f_\ell\in {\cal N}} \Big\{
    \frac{1}{2}\langle L\mathbf{f}^h,\mathbf{f}^h\rangle - A\big(f^h\big)\Big\}
\end{align}
when the target surface ${\cal N}$ is a unit disk $D\subset{\cal R}^2$, where $A\big(f^h\big)$ is the area of $f^h(S_V)$. Here the continuous conformal energy is discretized as $E_C(f^h) 
    = \frac{1}{2}\langle L\mathbf{f}^h,\mathbf{f}^h\rangle - A(f^h)$.
The CCEM solves (\ref{prog:DCE_S}) via quasi-Newton method, starting with an initial guess from a discrete weak solution of the inhomogeneous Laplace-Beltrami equation. The algorithm FDCP \cite{PTLM15} and LDCP \cite{GPLM18} are based on the discrete Beltrami equation to compute an estimated conformal mapping from the surface $\cal M$ to the unit disk in a couple form of two discrete quasi-conformal mappings. 

When a discrete algorithm is used for computing a discrete conformal mapping via minimizing a discrete conformal energy $\mathcal{E}_C(f^h)$, it is naturally assumed that (1) the discrete error is ignorable and, (2) the discrete error will tend to zero when the number of vertices tends to infinity and the mesh size $h$ tends to zero. Unfortunately, without conditions, 
\[
    \min_{f^h: f^h(S_V)\in{\cal N}}\lim_{h\to 0} \mathcal{E}_C(f^h)
    \neq \lim_{h\to 0} \min_{f^h: f^h(S_V)\in{\cal N}}\mathcal{E}_C(f^h).
\]
Even if the equality holds, it is also not clear whether the discrete solution tends to a conformal mapping as the mesh size $h$ tends to zero.
To our knowledge, no conditions are given in the literature for guaranteeing the convergence, except \cite{Hutc91} where an analysis on the error estimation was given when both $\cal M$ and $\cal N$ are subdomains on the plane, {\it i.e.}, $m=n=2$. Practically, if the discrete conformal mapping $f^h$ is computed via minimizing a discrete conformal energy $\mathcal{E}_C(f^h)$, two key questions should be addressed: 

\begin{question}
Given a discrete mapping $f^h$ from a given set of vertices $\{v_\ell\}\subset\cal M$ to $\cal N$, is there a continuous mapping $f: \cal M\to\cal N$ such that the error between the continuous conformal energy ${\cal E}_C(f)$ and the discrete conformal energy ${\cal E}_C(f^h)$, denoted as
\[
    \varepsilon_C(f^h) = \mathcal{E}_C(f) -\mathcal{E}_C(f^h),
\]
could be estimated? or how small is the discrete error of conformal energy?
\end{question}

\begin{question}
Assume that $f^h$ minimizes the discrete conformal energy ${\cal E}_C(f^h)$.
If the mesh size $h$ tends to zero, does the discrete error $\varepsilon_C(f^h)$ tend to zero? If yes, does $f^h$ converge to a conformal mapping?
\end{question}

In this paper, we try to address these two questions from the view of continuous mapping $f$, rather than from the view of piecewise linear mapping $f^h$. Practically, we will drive a discretization ${\cal E}_C^h(f)$ of the conformal energy ${\cal E}_C(f)$ for arbitrary continuous mapping $f:\cal M\to\cal N$. Here, ${\cal E}_C^h(f)$ differs from the discrete energy ${\cal E}_C(f^h)$ defined by a discrete mapping $f^h$ directly such as the objective function in (\ref{prog:DCE_S}), though a continuous mapping $f$ can also generate a discrete $f^h$ with $f^h(v_\ell) = f(v_\ell)$ at all the vertices.
We will give a careful analysis on the discrete error $\varepsilon_C^h(f) = \mathcal{E}_C(f) -\mathcal{E}_C^h(f)$. It consists of two estimations: 

(1) The error estimation of discretizing the smooth surface $\cal M$. The surface $\cal M$ is partitioned to a sequence pieces $\{{\cal M}_{ijk}\}$ corresponding to the triangulation $\{V_{ijk}\}$. We focus on the distance of an arbitrary point $\mathbf{x}\in{\cal M}_{ijk}$ to the triangle plane of $V_{ijk}$, together with the approximation of the triangle plane of $V_{ijk}$ to the tangent space of ${\cal M}$ at $\mathbf{x}$.  

(2) The error estimation of the discrete conformal energy based on the surface discretization. The function $\|\nabla\!_{\cal M}f(\mathbf{x})\|_F^2$ for $\mathbf{x}\in{\cal M}_{ijk}$ is approximated as a constant in terms of the values $f(v_i)$, $f(v_j)$, and $f(v_k)$, and the approximation error $\epsilon_{ijk}(\mathbf{x})$ is carefully bounded. Then, the continuous conformal energy is discretized with an error $\varepsilon_C^h(f)$ in a sum of all the integration errors over the pieces ${\cal M}_{ijk}$. 

The error analysis addresses Question 1 for the discrete mappings $\{f^h\}$ generated from continuous mappings $\{f\}$. Furthermore, we will prove that under a weak condition on the triangulation, the discrete error $\varepsilon_C^h(f)$ tends to zero as the mesh size $h\to 0$ for any smooth $f: {\cal M}\to \cal N$. More importantly,  we will prove that 
\[
    \min_{f({\cal M})={\cal N}} \lim_{h\to 0}\mathcal{E}_C^h(f)
    = \lim_{h\to 0} \min_{f({\cal M})={\cal N}}\mathcal{E}_C^h(f),
\]
which addresses Question 2 perfectly. The condition is weaker than the quasi-uniform condition given in \cite{Hutc91} for the convergence of discrete conformal mapping between two 2D sets.

{\bf Notation}. 
${\cal A}(\cdot)$ denotes the area of a predicted smooth surface as $\cal M$, while $A(\cdot)$ is the area of triangle surface as $S_V$. 
Given three vertices $v_i,v_j,v_k$ in a space, $V_{ijk}$ is the triangle and $v_{ijk} = [v_i,v_j,v_k]$ is a three order matrix. We also use $v_{ij} = v_i-v_j$ as a directed edge from $v_i$ to $v_j$ for simplicity. $d(V_{ijk}) = \max\{\|v_{ij}\|,\|v_{jk}\|,\|v_{ki}\|\}$ is the diameter of $V_{ijk}$ in Euclidean norm.
The angle of $V_{ijk}$ opposite the edge $v_{ij}$ is denoted by $\beta_{ij}$. Notice that $\beta_{ji}$ is the angle opposite to the edge $v_{ji}$ in the triangle $V_{jik'}$.

\section{Error Analysis of Surface Discretization}
\label{sec:discrete surface}

Suppose $\cal M$ is a smooth simple surface embedded in ${\cal R}^m$ with $m\geq 2$ and it can be parameterized as a column  vector function $\mathbf{x} = \mathbf{x}(\omega)$ with $\omega\in \Omega\subset{\cal R}^2$. 
Let $\mathcal{V} = \{v_1, \cdots, v_N\}$ be a set of vertices sampled from $\cal M$, {\it i.e.}, $\{v_i = \mathbf{x}(\omega_i)\}$, and let $V_{ijk}$ be an arbitrary triangles of the vertices $v_i,v_j,v_k$ not containing any other vertices inside, and let $\pi(V_{ijk})$ be the plane containing $V_{ijk}$. The vertices $v_i,v_j,v_k$ has the anticlockwise order in the right-hand system corresponding to the predicted normal vector $\mathbf{n}$. The triangle in the parameter space, corresponding to $V_{ijk}$, is denoted as $\Omega_{ijk}$ with the 2D vertices $\omega_i,\omega_j,\omega_k$. Furthermore, each $\Omega_{ijk}$ also determines a subsurface ${\cal M}_{ij} = \big\{\mathbf{x}(\omega):\ \omega\in\Omega_{ijk}\big\}$ of $\cal M$. Hence,  
${\cal M}$ could be partitioned as a union of $\{{\cal M}_{ijk}\}$ without overlaps. 

In this section, we give an error bound between the surface $\cal M$ and the triangle surface $S_V$, focusing on the distance of an arbitrary $\mathbf{x}\in{\cal M}_{ijk}$ to its orthogonal projection onto the triangle plane $\pi(V_{ijk})$, and the error between an arbitrary tangent space of ${\cal M}_{ijk}$ and $\pi(V_{ijk})$. For simplicity, we will use 
\[
    T_{ijk} = A(\Omega_{ijk}), \quad A_{ijk} = A(V_{ijk}),\quad
    {\cal A}_{ijk} = {\cal A}({\cal M}_{ijk}).
\]
For an arbitrary point $\mathbf x\in\cal M$, $Q(\mathbf{x})=\nabla \mathbf{x}
    (\nabla \mathbf{x}^\top\nabla \mathbf{x})^{-1/2}$ is an orthogonal basis matrix.

\subsection{Barycentric Coordinate}\label{sec:barycentric cond}

Let us consider an arbitrary point $\mathbf{x}\in \mathcal{M}_{ijk}$ and let $p(\mathbf{x})$ be its orthogonal projection onto the triangle plane $\pi(V_{ijk})$. The projected point can be represented as a linear combination of the vertices $v_i,v_j,v_k$,
\begin{align}\label{def:p}
    p = a_iv_i + a_jv_j + a_kv_k = v_{ijk}a,
    \quad a_i + a_j + a_k = 1,
\end{align}
where $a_i,a_j,a_k$ are the barycentric coordinates of $p$, uniquely depends on $p$, and nonnegative if $p\in V_{ijk}$ \cite{DRoland08}.
Practically, let $\mathbf{n}_{ijk}$ be the unit normal vector of $\pi(V_{ijk})$ in a right-hand system. It is known that the exterior products between the directed edges $v_{ij}$, $v_{jk}$, and $v_{ki}$ satisfies
\[
    v_{ij}\times v_{jk} = v_{jk}\times v_{ki} = v_{jk}\times v_{ij} = A_{ijk}\mathbf{n}_{ijk}. 
\]
By $a_i + a_j + a_k = 1$, $p-v_j= a_iv_{ij}-a_kv_{jk}$. Thus, the inner product between $s_i = v_{jk}\times \mathbf{n}_{ijk}$ and $p-v_j$ is 
\[
    \langle s_i,p-v_j\rangle
    = a_i \langle v_{jk}\times\mathbf{n}_{ijk},v_{ij}\rangle
    = a_i \langle \mathbf{n}_{ijk},v_{ij}\times v_{jk}\rangle
    = 2A_{ijk}a_i.
\]
Similarly, $2A_{ijk}a_j = \langle s_j,p-v_k\rangle$ and $2A_{ijk}a_k = \langle s_k,p-v_i\rangle$, where $ s_j=v_{ki}\times \mathbf{n}_{ijk}$ and $s_k=v_{ij}\times \mathbf{n}_{ijk}$.
Hence, the barycentric coordinates are uniquely determined as
\begin{align}\label{a}
    a_i 
    = \langle b_i,p-v_j\rangle,\quad
    a_j 
    = \langle b_j,p-v_k\rangle,\quad
    a_k 
    = \langle b_k,p-v_i\rangle,
\end{align}
where $b_i = \frac{s_i}{2A_{ijk}}$, $b_j = \frac{s_j}{2A_{ijk}}$, $b_k = \frac{s_k}{2A_{ijk}}$. 

The simple representation shows that $a = (a_i,a_j,a_k)^\top$ depends on $p$ smoothly. Taking $a = a(p)$ as a mapping from $\pi(V_{ijk})$ to ${\cal R}^3$, its gradient matrix with respect to $p$ is 
\begin{align}\label{nabla a}
    \nabla a(p) 
    = b_{ijk}^\top, \quad b_{ijk} = [b_i,b_j,b_k].
\end{align}
On the other hand, as an orthogonal projection of $\mathbf{x}\in\cal M$ onto the triangle plane, $p=p(\mathbf{x})$ depends on $\mathbf{x}$ continuously. Hence, the projection can be represented as 
\begin{align}\label{p}
    p(\mathbf{x}) =\mathbf{x}-\tau(\mathbf{x})\mathbf{n}_{ijk},\quad
    \tau(\mathbf{x}) = \mathbf{n}_{ijk}^\top\big(\mathbf{x}-p(\mathbf{x})\big).
\end{align}
Since we can also represent
$p(\mathbf{x}) = p_0+(I-\mathbf{n}_{ijk}\mathbf{n}_{ijk}^\top)(\mathbf{x}-p_0)$ with a fixed point $p_0\in \pi(V_{ijk})$, we have
\begin{align}\label{nabla p}
    \nabla p(\mathbf{x}) = I-\mathbf{n}_{ijk}\mathbf{n}_{ijk}^\top,\quad
    \nabla \tau(\mathbf{x}) = \mathbf{n}_{ijk}^\top\big(I-\nabla p(\mathbf{x})\big)
    =\mathbf{n}_{ijk}^\top.
\end{align}
Thus, barycentric coordinate vector $a = a\big(p(\mathbf{x})\big)$, as a compounded mapping from $\mathbf{x}$ to ${\cal R}^3$, has the gradient with respect to $\mathbf{x}$
as
\begin{align}\label{nabla ax}
    \nabla a(p(\mathbf{x})) 
    = \nabla a(p) \nabla p(\mathbf{x})
    = b_{ijk}^\top, \quad
    \forall \mathbf{x}\in {\cal M}_{ijk},
\end{align}
the same as the gradient with respect to $p$ since $s_{\ell}^\top \mathbf{n}_{ijk} = 0$ by (\ref{nabla a}).

In the next two subsections, we will give two upper bounds for the distance $\tau = \tau(\mathbf{x})$ and the error between the tangent space of ${\cal M}_{ijk}$ at $\mathbf{x}$ and the triangle plane $\pi(V_{ijk})$, measured by $\mathbf{n}_{ijk}^\top Q(\mathbf{x})$ with the orthogonal basis of the tangent space $Q(\mathbf{x}) =\nabla \mathbf{x}(\nabla \mathbf{x}^\top\nabla \mathbf{x})^{-1/2}$. In the next subsection we will estimate $|\tau(\mathbf{x})|$ and $\|\mathbf{n}_{ijk}^\top Q(\mathbf{x})\|_2$, based on the representation (\ref{nabla ax}).

\subsection{The projection distance and tangent space error}

We need some preparing lemmas for getting the error bounds. The first one gives a simple upper bound of $\tau(\mathbf{x})$, assuming that the gradient $\nabla\mathbf{x}(\omega)$ is Lipschitz continuous.   
\begin{align}\label{L-cond:M}
    \|\nabla\mathbf{x}(\omega)-\nabla\mathbf{x}(\omega')\|_F
    \leq C_{\cal M}\|\omega-\omega'\|.
\end{align}

\begin{lemma}\label{lma:tau}
If $\nabla\mathbf{x}(\omega)$ is Lipschitz continuous with Lipschitz constant $C_{\cal M}$, then
\begin{align}\label{bound:tau}
    \max_{\omega\in\Omega_{ijk}}|\tau(\mathbf{x})|\leq 
    C_{\cal M}d^2(\Omega_{ijk}).
\end{align}
\end{lemma}
\begin{proof}
Consider the extension of $\mathbf{x}= \mathbf{x}(\omega)$ for $\omega\in\Omega_{ijk}$, $v_\ell = \mathbf{x}(\omega_\ell) 
    = \mathbf{x}+\nabla \mathbf{x}(\omega'_\ell)(\omega_\ell-\omega)$. We rewrite
\begin{align}\label{v_ell}
    v_\ell = \mathbf{x}(\omega_\ell) 
    = \mathbf{x}+\nabla \mathbf{x}(\omega)(\omega_\ell-\omega)+g_\ell,
\end{align}
where $g_\ell = \big(\nabla \mathbf{x}(\omega'_\ell)-\nabla \mathbf{x}(\omega)\big)(\omega_\ell-\omega)$. By the Lipschitz condition,
\begin{align}\label{g}
    \|g_\ell\|_2 
    \leq C_{\cal M}\|\omega'_\ell-\omega\|\|\omega_\ell-\omega\|
    \leq C_{\cal M}d^2(\Omega_{ijk}).
\end{align}

Furthermore, let $\alpha_i,\alpha_j,\alpha_k$ be the barycentric coordinates of $\omega\in \Omega_{ijk}$,
{\it i.e.}, $\omega = \alpha_i\omega_i+\alpha_j\omega_j+\alpha_k\omega_k$, and let
$\tilde p = \alpha_iv_i+\alpha_jv_j+\alpha_kv_k\in V_{ijk}$. By (\ref{v_ell}), $\tilde p- \mathbf{x}=\sum \alpha_\ell g_\ell$. Hence, using (\ref{g}),
\begin{align}\label{tau}
    |\tau(\mathbf{x})| 
    =\min_{p'\in \pi(V_{ijk})}\|p'-\mathbf{x}\|_2
    \leq\|\tilde p- \mathbf{x}\|_2
    \leq \sum_\ell\alpha_\ell\|g_\ell\|_2
    \leq C_{\cal M}d^2(\Omega_{ijk}),
\end{align}
completing the proof.
\end{proof}

For bounding $\mathbf{n}_{ijk}^\top Q(\mathbf{x})$ in norm, we also use the decomposition (\ref{p}) and (\ref{v_ell}) to get
\[
    v_{ijk}-pe^\top-\tau(\mathbf{x})\mathbf{n}_{ijk}e^\top
    = v_{ijk}-\mathbf{x}e^\top
    = \nabla \mathbf{x}\hat\omega_{ijk}+g_{ijk},
\]
where $g_{ijk} = [g_i,g_j, g_k]$ and 
$\hat\omega_{ijk}=\hat\omega_{ijk}(\omega) = \omega_{ijk}-\omega e^\top$. 
By 
\[
    \mathbf{n}_{ijk}^\top\nabla \mathbf{x}\hat\omega_{ijk}
    = -\tau(\mathbf{x})e^\top-\mathbf{n}_{ijk}^\top g_{ijk},
\]
since $\mathbf{n}_{ijk}^\top(v_{ijk}-pe^\top)=0$,
hence, for the orthogonal basis $Q(\mathbf{x})=\nabla \mathbf{x}
    (\nabla \mathbf{x}^\top\nabla \mathbf{x})^{-1/2}$, 
\[
    \mathbf{n}_{ijk}^\top Q(\mathbf{x})
    = \mathbf{n}_{ijk}^\top \nabla \mathbf{x}
        (\nabla \mathbf{x}^\top\nabla\mathbf{x})^{-1/2}
    = -\big(\tau(\mathbf{x})e^\top+\mathbf{n}_{ijk}^\top g_{ijk}\big)
    \hat\omega_{ijk}^{\dag} (\nabla \mathbf{x}^\top\nabla\mathbf{x})^{-1/2}.
\]
Since $\|(\nabla \mathbf{x}^\top\nabla\mathbf{x})^{-1/2}\|_2 = 1/\sigma_{\min}(\nabla\mathbf{x})$, $\|g_{ijk}\|_2\leq \sqrt{3}C_{\cal M}d^2(\Omega)$,
and by the bound (\ref{tau}), we get
\begin{align}\label{nQ}
    \|\mathbf{n}_{ijk}^\top Q(\mathbf{x})\|_2
    &\leq \frac{2\sqrt{3}C_{\cal M}}{\sigma_{\min}(\nabla\mathbf{x})}
        d^2(\Omega_{ijk})\|\hat\omega_{ijk}^{\dag}\|_2.
\end{align}

Furthermore, we need the lemma to bound $\|\hat\omega_{ijk}^{\dag}\|_2$ that is the inverse of square root of the smallest eigenvalue of $(\omega_{ijk}-\omega e^\top)(\omega_{ijk}-\omega e^\top)^\top$. The following lemma gives a general result for the eigenvalues of $B(x) = (C-xe^\top)(C-xe^\top)^\top$ for arbitrary matrix $C\in {\cal R}^{m\times n}$ and $x\in {\cal R}^n$. 

\begin{lemma}\label{lma:eig}
Let $C \in {\cal R}^{m\times n}$ and $B(x) = (C-xe^\top)(C-xe^\top)^\top$ for arbitrary $x\in {\cal R}^n$. Then each eigenvalue $\lambda_\ell(x)$ of $B(x)$ achieves its minimum at the mean $c_*$ of columns of $C$.
\end{lemma}
\begin{proof}
Let $\lambda_1(x)\leq \cdots\leq \lambda_m(x)$. It is known that for each $\ell$, there is a differentiable unit eigenvector $u_\ell=u_\ell(x)\in {\cal R}^n$ of $B(x)$ with respect to $x$, corresponding to $\lambda_\ell = \lambda_\ell(x)$. Taking derivatives on the two sides of the eigen-equation $B(x)u_\ell(x) =\lambda_\ell(x) u_\ell(x)$ with respect to each component $x_t$ of $x$, we get that
\begin{align}\label{deriv_eig}
    \frac{\partial B}{\partial x_t} u_\ell
    +B\frac{\partial u_\ell}{\partial x_t}
    =\frac{\partial \lambda_\ell}{\partial x_t}u_\ell
    +\lambda_\ell \frac{\partial u_\ell}{\partial x_t}.
\end{align}

Let $x_*^{(\ell)}$ be the minimizer of $\lambda_\ell(x)$. 
Taking the derivative with respect to the component $x_t$ of $x$ on the two sides of the eigen-equation $B(x)u_\ell(x) =\lambda_\ell(x) u_\ell(x)$, and using $\frac{\partial \lambda_\ell}{\partial x}\big|_{x=x_*^{(\ell)}} = 0$, we obtain that $u_\ell^\top\frac{\partial B}{\partial x_t} u_\ell\big|_{x=x_*^{(\ell)}} =0$. Since $\frac{\partial B}{\partial x_t}=m\big(e_t (x-c_*)^\top+(x-c_*)e_t^\top\big)$, the equation $u_\ell^\top\frac{\partial B}{\partial x_t} u_\ell\big|_{x=x_*^{(\ell)}} =0$ becomes  
\[
   (x_*^{(\ell)}-c_*)^\top u_\ell= 0,\quad
   {\rm where}\quad u_\ell = u_\ell(x_*^{(\ell)}).
\]
Hence, by $B(x) = CC^\top -mc_*x^\top+mx(x-c_*)^\top$, we see that
\[
    \lambda_\ell(x_*^{(\ell)}) u_\ell
    = B(x_*^{(\ell)})u_\ell 
    = (CC^\top -mc_*c_*^\top)u_\ell
    = B(c_*)u_\ell.
\]
It implies that each minimum $\lambda_\ell(x_*^{(\ell)})$ is also an eigenvalue of $B(c_*)$ and $u_\ell = u_\ell(x_*^{(\ell)})$ is the corresponding eigenvector. 
Finally, 
\[
    \lambda_{\ell-1}(c_*^{(\ell-1)})
    = \min_x\lambda_{\ell-1}(x)
    \leq \lambda_{\ell-1}(c_*^{(\ell)})
    \leq \lambda_\ell(c_*^{(\ell)}), \quad \ell=2,3,\cdots,m.
\]
That is, as $\lambda_1(c_*)\leq \cdots\leq \lambda_m(c_*)$, $\lambda_1(c_*^{(1)})\leq \lambda_2(c_*^{(2)})\leq \cdots\leq \lambda_m(c_*^{(m)})$ are also the $m$ ordered eigenvalues of $B(c_*)$. Hence, $\lambda_\ell(c_*^{(\ell)}) = \lambda_\ell(c_*)$ for each $\ell$.
\end{proof}

Applying Lemma \ref{lma:eig} on $B=\hat\omega_{ijk}(\omega)\hat\omega_{ijk}^\top(\omega)$ 
and using $\|\hat\omega_{ijk}^{\dag}(\omega)\|_2 = \frac{1}{\sqrt{\lambda_1(\omega)}}$, 
where $\lambda_1(\omega)$ is the smallest eigenvalue of $\hat\omega_{ijk}(\omega)\hat\omega_{ijk}^\top(\omega)$, we have 
\begin{align}\label{bound:hat_omega}
    \|\hat\omega_{ijk}^{\dag}(\omega)\|_2 
    = \frac{1}{\sqrt{\lambda_1(\omega)}}
    \leq\frac{1}{\sqrt{\lambda_1(\omega_*)}}
    = \|\hat\omega_{ijk}^{\dag}(\omega_*)\|_2,
\end{align}
where $\omega_*=\frac{1}{3}\omega_{ijk}e$ by Lemma \ref{lma:eig}. Furthermore, $\|\hat\omega_{ijk}^{\dag}(\omega_*)\|_2$ has a simple form.

\begin{lemma}\label{lma: hat_omega_{ijk}}
For the triangle $\Omega_{ijk}$ with vertices $\omega_i,\omega_j,\omega_k$,
\begin{align}\label{maxW}
    \max_{\omega\in\Omega_{ijk}}\|\hat\omega_{ijk}^{\dag}(\omega)\|_2
    =\frac{\|[\omega_{jk},\omega_{ki},\omega_{ij}]\|_2}{2T_{ijk}}.
\end{align}
\end{lemma}
\begin{proof}
Since $\omega_i-\omega_* = \frac{1}{3}(2\omega_i-\omega_j-w_k) = \frac{1}{3}(2\omega_{ij}+\omega_{jk})$, and similarly,
$\omega_j-\omega_* = \frac{1}{3}(\omega_{jk}-\omega_{ij})$ and $\omega_k-\omega_*=-\frac{1}{3}(\omega_{ij}+2\omega_{jk})$,
$\hat\omega_{ijk}(\omega_*)$ has the representation
\begin{align*}
    \hat\omega_{ijk}(\omega_*) 
    = \frac{1}{3}[\omega_{ij},\omega_{jk}]
    \left[\begin{array}{rrr}
         2 & -1 & -1\\
         1 &  1 & -2
    \end{array}\right].
\end{align*}
Hence, $\big(\hat\omega_{ijk}^{\dag}(\omega_*)\big)^\top=(\hat\omega_{ijk}(\omega_*)\hat\omega_{ijk}^\top(\omega_*))^{-1}\hat\omega_{ijk}(\omega_*)$ can be represented as
\begin{align*}
    &\ \big(\hat\omega_{ijk}^{\dag}(\omega_*)\big)^\top 
    = [\omega_{ij},\omega_{jk}]^{-\top}
    \left[\begin{array}{rr}
         2 & 1 \\
         1 & 2
    \end{array}\right]^{-1}
    \left[\begin{array}{rrr}
         2 & -1 & -1\\
         1 &  1 & -2
    \end{array}\right]
    =[\omega_{ij},\omega_{jk}]^{-\top}
    \left[\begin{array}{rrr}
         1 & -1 & 0\\
         0 &  1 & -1
    \end{array}\right]\\
    =&\ \frac{1}{2\det[\omega_{ij},\omega_{jk}]}
    \left[\begin{array}{rr}
         0 & 1 \\
         -1 & 0
    \end{array}\right][\omega_{ij},\omega_{jk}]
    \left[\begin{array}{rrr}
         0 & -1 & 1\\
         1 & -1 & 0
    \end{array}\right]
    =\frac{1}{2T_{ijk}}
    \left[\begin{array}{rr}
         0 & 1 \\
         -1 & 0
    \end{array}\right][\omega_{jk},\omega_{ki},\omega_{ij}].
\end{align*}
Here we have used the equality $C^{-\top} = \frac{1}{\det C}\big({0\atop -1}\ {1\atop 0}\big)C\big({0\atop 1}\ {-1\atop 0}\big)$
for a nonsingular matrix $C$ of order 2, and $\det[\omega_{ij},\omega_{jk}]=2T_{ijk}$. Combining it with (\ref{bound:hat_omega}), we get the bound (\ref{maxW}) immediately.
\end{proof}

By Lemma \ref{lma: hat_omega_{ijk}}, $\|\hat\omega_{ijk}^{\dag}(\omega)\|_2 \leq \frac{\|[\omega_{jk},\omega_{ki},\omega_{ij}]\|_2}{2T_{ijk}}\leq \frac{\sqrt{3}d(\Omega_{ijk})}{2T_{ijk}}$. Substituting the bound into (\ref{nQ})
we get an upper bound of $\|\mathbf{n}_{ijk}^\top Q(\mathbf{x})\|$. 

\begin{lemma}\label{lma:nQ}
If $\nabla\mathbf{x}(\omega)$ is Lipschitz continuous with Lipschitz constant $C_{\cal M}$, then
\begin{align}\label{bound:nQ}
    \|\mathbf{n}_{ijk}^\top Q(\mathbf{x})\|
    \leq \frac{3C_{\cal M}}{\sigma_{\min}(\nabla\mathbf{x})}
        \frac{d^3(\Omega_{ijk})}{T_{ijk}}.
\end{align}
\end{lemma}

\section{Approximation of discrete conformal energy}

The classical discretization of continuous conformal energy is obtained via using the piecewise linear surface $\{V_{ijk}\}$ as an approximation of the surfaces ${\cal M}$, where the surface gradient $\nabla_{\cal M}f$ is approximated as piecewise constants.
In this section, we first give a careful estimation to the discrete error of the surface gradient $\nabla_{\cal M}f$ given $f$ in terms of $\|\mathbf{n}_{ijk}^\top Q(\mathbf{x})\|$. Combining it with the error analysis given in the previous sections, the discrete error of continuous conformal energy follows, 

\subsection{Discrete estimation of surface gradient}

Consider the following functions for $\mathbf{x}\in{\cal M}_{ijk}$
\begin{align}\label{f_ell}
    h_{\ell}(\mathbf{x}) 
    = f(v_\ell) - f(\mathbf{x}) + \nabla\!_{\cal M}f(\mathbf{x})(\mathbf{x}-v_\ell), 
    \quad \ell = i,j,k.
\end{align}
Taking any convex combination of them with weight vector $\alpha=\{\alpha_\ell\}$, we have
\begin{align}\label{f}
    f(\mathbf{x})
    =\big(f_{ijk}-h_{ijk}(\mathbf{x})\big)\alpha 
    +\nabla\!_{\cal M}f(\mathbf{x})(\mathbf{x}-v_{ijk}\alpha),
\end{align}
where $f_{ijk} = [f_i,f_j,f_k]$, $h_{ijk}(\mathbf{x}) = [h_i(\mathbf{x}),h_j(\mathbf{x}),h_k(\mathbf{x})]$, and $v_{ijk} = [v_i,v_j,v_k]$.

\begin{lemma}\label{lma:nabla Mf}
Assume that $f$ is $C^1$-continuous over ${\cal M}$. Then for any $\mathbf{x}\in{\cal M}_{ijk}$, 
\begin{align}\label{nabla Mf}
    \nabla\!_{\cal M}f(\mathbf{x})
    &=\big(f_{ijk}-h_{ijk}(\mathbf{x})\big)b_{ijk}^\top
        +\nabla\!_{\cal M}f(\mathbf{x})\mathbf{n}_{ijk}\mathbf{n}_{ijk}^\top.
\end{align}
\end{lemma}
\begin{proof}
Let $u$ be an arbitrary vector and $t\neq 0$ a scale such that 
$p_0 = p-t(I-\mathbf{n}_{ijk}\mathbf{n}_{ijk}^\top)u\in V_{ijk}$.
Choose $\alpha = a(p)$ and $\alpha = a(p_0)$, the vectors of the barycentric coordinates of $p$ and $p_0$, respectively, in (\ref{f}), we have two representations of $f(\mathbf{x})$. Combining the two representations, we get 
\begin{align}\label{p-p_0}
    \big(f_{ijk}-h_{ijk}(\mathbf{x})\big)\big(a(p)-a(p_0)\big)
    -\nabla\!_{\cal M} f(\mathbf{x})(p-p_0)=0.
\end{align}
Since $a(p)$ is a linear function of $p$ by (\ref{nabla a}), we have
$a(p) - a(p_0) = \nabla a(p)(p-p_0) = b_{ijk}^\top(p-p_0)$.
Substituting it into (\ref{p-p_0}), we get
\[
    0 = \Big\{\big(f_{ijk}-h_{ijk}(\mathbf{x})\big)b_{ijk}^\top
        -\nabla\!_{\cal M} f(\mathbf{x}) \Big\}(p-p_0)
    =t\Big\{\big(f_{ijk}-h_{ijk}(\mathbf{x})\big)b_{ijk}^\top
        -\nabla\!_{\cal M} f(\mathbf{x}) \Big\} (I-\mathbf{n}_{ijk}\mathbf{n}_{ijk}^\top)u.
\]
Since $t\neq 0$ and $u$ is arbitrary, we conclude that 
\[
    \Big\{\big(f_{ijk}-h_{ijk}(\mathbf{x})\big)b_{ijk}^\top
        -\nabla\!_{\cal M} f(\mathbf{x})\Big\}
    (I-\mathbf{n}_{ijk}\mathbf{n}_{ijk}^\top) = 0.
\]
By definition, $b_{ijk}^\top\mathbf{n}_{ijk} = 0$. Hence,
\[
    \nabla\!_{\cal M} f(\mathbf{x})(I-\mathbf{n}_{ijk}\mathbf{n}_{ijk}^\top)
    = \big(f_{ijk}-h_{ijk}(\mathbf{x})\big)b_{ijk}^\top
    (I-\mathbf{n}_{ijk}\mathbf{n}_{ijk}^\top)
    =\big(f_{ijk}-h_{ijk}(\mathbf{x})\big)b_{ijk}^\top.
\]
This is equivalent to (\ref{nabla Mf}).
\end{proof}

Based on (\ref{nabla Mf}), the continuous function $\|\nabla_{\!\cal M} f(\mathbf{x})\|_F^2$ can be approximated by the discrete $\big\|f_{ijk}b_{ijk}^\top\big\|_F^2$ for $\mathbf{x}\in{\cal M}_{ijk}$. The following theorem gives an upper bound of the approximation error.

\begin{theorem}\label{thm:error esti}
Suppose $f$ is $C^1$-continuous over $\cal M$, then for $\mathbf{x}\in{\cal M}_{ijk}$,
\begin{align}\label{error}
    \Big|\|\nabla_{\!\cal M} f(\mathbf{x})\|_F^2     
        -\big\|f_{ijk}b_{ijk}^\top\big\|_F^2\Big|
    \leq\psi_f^2(\mathbf{x})+2\|f_{ijk}b_{ijk}^\top\|_F\psi_f(\mathbf{x}),
\end{align}
where 
\begin{align}\label{psi_f}
    \psi_f(\mathbf{x})
    =\|h_{ijk}b_{ijk}^\top\|_F
    +\|\nabla\!_{\cal M}f(\mathbf{x})\|_F\|\mathbf{n}_{ijk}^\top Q\|_2^2.
\end{align}
\end{theorem}
\begin{proof}
We rewrite (\ref{nabla Mf}) as
$
    \nabla\!_{\cal M}f(\mathbf{x}) 
    = \nabla\!_{\cal M}f(\mathbf{x})QQ^\top
    = f_{ijk}b_{ijk}^\top QQ^\top +\psi_1Q^\top+\psi_2Q^\top,
$
where 
\[
    \psi_1 = -h_{ijk}b_{ijk}^\top Q,\quad 
    \psi_2 = \nabla\!_{\cal M}f(\mathbf{x})\mathbf{n}_{ijk}\mathbf{n}_{ijk}^\top Q
    =\nabla\!_{\cal M}f(\mathbf{x})QQ^\top\mathbf{n}_{ijk}\mathbf{n}_{ijk}^\top Q.
\]
Hence,
$
    \|\nabla\!_{\cal M}f(\mathbf{x})\|_F^2
    = \|f_{ijk}b_{ijk}^\top Q\|_F^2+\|\psi_1+\psi_2\|_F^2
        +2\langle f_{ijk}b_{ijk}^\top Q,\psi_1+\psi_2\rangle
$, and 
\begin{align*}
    \Big|\|\nabla_{\!\cal M} f(\mathbf{x})\|_F^2     
        -\big\|f_{ijk}b_{ijk}^\top\big\|_F^2\Big|
    &\leq \|\psi_1+\psi_2\|_F^2
         +2\|f_{ijk}b_{ijk}^\top\|_F \|\psi_1+\psi_2\|_F\\
    &\leq (\|\psi_1\|_F+\|\psi_2\|_F)^2
         +2\|f_{ijk}b_{ijk}^\top\|_F (\|\psi_1\|_F+\|\psi_2\|_F).
\end{align*}
Hence. (\ref{error}) follows from 
$\|\psi_1\|_F\leq \|h_{ijk}b_{ijk}^\top\|_F$ and
$\|\psi_2\|_F\leq \|\nabla\!_{\cal M}f(\mathbf{x})\|_F
    \|\mathbf{n}_{ijk}^\top Q\|_2^2$.
\end{proof}

The dominant term in the error representation (\ref{error}) has the simple form by (\ref{nabla ax}), 
\begin{align}\label{fb}
    \big\|f_{ijk}b_{ijk}^\top\big\|_F^2=
    \frac {1}{4A_{ijk}^2} \big\|f_{ijk}s_{ijk}^\top\big\|_F^2
    =\frac {1}{2A_{ijk}} \sum_{\ell m\in\{ij,jk,ki\}}
        \|f(v_\ell)-f(v_m)\|^2\cot\beta_{\ell m},
\end{align}
where the last equality is given in \cite{XGST20book}. 

Under some assumptions,  $\|\psi_1\|_F$ and $\|\psi_2\|_F$ can be will further bounded. For instance, by Lemma \ref{lma:nQ}, if $\nabla\mathbf{x}$ is Lipschitz continuous, then 
\[
	\|\psi_2\|_F\leq \|\nabla\!_{\cal M}f(\mathbf{x})\|_F\Big\{
		\frac{3C_{\cal M}d^3(\Omega_{ijk})}{\sigma_{\min}(\nabla\mathbf{x})T_{ijk}} \Big\}^2. 
\]
We also need the Lipschitz condition for $\nabla_{\cal M}f(\mathbf{x})$ 
\begin{align}\label {L-cond:f}
	\big\|\nabla_{\cal M}f(\mathbf{x})-\nabla_{\cal M}f(\hat{\mathbf{x}})\big\|_F
	\leq c_L(f)\|\mathbf{x}-\hat{\mathbf{x}}\|
\end{align}
to bound the factor $\|h_{ijk}\|_F$ in the bound $\|\psi_1\|_F\leq \|h_{ijk}\|_F\|b_{ijk}\|_2$. To this end, we use the first order estimation of $f(\mathbf{x}(\omega))$,
\[
    f(\mathbf{x}) = f(\mathbf{x}(\omega)) = 
    f(v_\ell)+ \nabla\!_{\omega}f(\mathbf{x}(\hat\omega))(\omega-\omega_\ell),
\]
where $\hat\omega\in\Omega_{ijk}$. Since $\nabla\!_{\omega}f(\mathbf{x}(\omega)) 
= \nabla_{\cal M}f(\mathbf{x})\nabla \mathbf{x}(\omega) $, 
\begin{align*}
	h_{\ell}(\mathbf{x}) 
	&= \nabla_{\cal M}f(\mathbf{x})(\mathbf{x}-v_\ell)-\nabla\!_{\omega}f(\mathbf{x}(\hat\omega))(\omega-\omega_\ell)\\
	&= \nabla_{\cal M}f(\mathbf{x})(\mathbf{x}-v_\ell)-\nabla\!_{\omega}f(\mathbf{x}(\omega))(\omega-\omega_\ell)
	+\big(\nabla\!_{\omega}f(\mathbf{x}(\omega)) - \nabla\!_{\omega}f(\mathbf{x}(\hat\omega)\big)(\omega-\omega_\ell)\\
	&= -\nabla_{\cal M}f(\mathbf{x})g_\ell
	+\big(\nabla\!_{\omega}f(\mathbf{x}(\omega)) - \nabla\!_{\omega}f(\mathbf{x}(\hat\omega)\big)(\omega-\omega_\ell).
\end{align*}
Here we have used (\ref{v_ell}) in the last equality. Furthermore, we also have 
\begin{align*}
	\nabla\!_{\omega}f(\mathbf{x}(\omega)) - \nabla\!_{\omega}f(\mathbf{x}(\hat\omega))
	&= \nabla_{\cal M}f(\mathbf{x})\nabla \mathbf{x}(\omega) 
	- \nabla_{\cal M}f(\hat{\mathbf{x}})\nabla \mathbf{x}(\hat\omega)\\
	&=\nabla_{\cal M}f(\mathbf{x})\big(\nabla \mathbf{x}(\omega) -\nabla \mathbf{x}(\hat\omega) \big)
	+\big(\nabla_{\cal M}f(\mathbf{x})-\nabla_{\cal M}f(\hat{\mathbf{x}})\big)\nabla \mathbf{x}(\hat\omega).
\end{align*}
Combining it with Lipschitz conditions (\ref {L-cond:M}) and (\ref{L-cond:f}), together with (\ref{g}), we get 
\begin{align}\label{bound_h}
	\|h_\ell\| 
	&\leq \big\{2C_{\cal M}\| \nabla_{\cal M}f(\mathbf{x})\|_Fd(\Omega_{ijk})
	+c_L(f)\|\mathbf{x}-\hat{\mathbf{x}}\|\|\nabla \mathbf{x}(\hat\omega)\|_F\big\}\|\omega-\omega_\ell\|
	\leq \epsilon_f,
\end{align}
and $\|h_{ijk}\|_F \leq \sqrt{3}\epsilon_f$, where 
\begin{align}\label{epsilon_f}	
	\epsilon_f=  \big\{2C_{\cal M}\| \nabla_{\cal M}f(\mathbf{x})\|_F
	+c_L(f)\max_{\omega}\|\nabla \mathbf{x}(\omega)\|_F^2\big\}d^2(\Omega_{ijk}).
\end{align}
Finally, by the definition of $\{s_\ell\}$ in Subsection \ref{sec:barycentric cond}, we write $s_{ijk} = G[v_{jk},v_{ki},v_{ij}]$ with a rotation matrix $G$. Hence,
$
    \|b_{ijk}\|_2 
    = \frac{1}{2A_{ijk}}\|s_{ijk}\|_2
    = \frac{1}{2A_{ijk}}\|[v_{ij},v_{jk},v_{ki}]\|_2
    \leq \frac{\sqrt{3}}{2A_{ijk}}d(V_{ijk})
$ and 
\begin{align}\label{bound:fHS}
	\|h_{ijk}b_{ijk}^\top\|_F \leq \frac{3d(V_{ijk})}{2A_{ijk}}\epsilon_f.
\end{align}
Therefore, we can further bound $\psi_f$ in the error estimation given in Theorem \ref{thm:error esti},
\begin{lemma}\label{lma:psi_f}
Assume that the Lipschitz conditions (\ref {L-cond:M}) and (\ref {L-cond:f}) are satisfied, then
\begin{align}\label{bound:psi_f}
  	 \psi_f(\mathbf{x})
	 &\leq \frac{3d(V_{ijk})}{2A_{ijk}}\epsilon_f
	 	+\|\nabla\!_{\cal M}f(\mathbf{x})\|_F\Big\{
		\frac{3C_{\cal M}d^3(\Omega_{ijk})}{\sigma_{\min}(\nabla\mathbf{x})T_{ijk}} \Big\}^2.
\end{align}
\end{lemma}

\subsection{Error estimation of conformal energy discretization}

Using the partitioning ${\cal M} = \bigcup {\cal M}_{ijk}$, the continuous Dirichlet energy 
is the sum of all the sub energy restricting the integral over each piece ${\cal M}_{ijk}$ of $\cal M$,
\[
    \mathcal{E}_D(f) = \sum_{{\cal M}_{ijk}} 
    \frac{1}{2}\int_{{\cal M}_{ijk}} \|\nabla_{\cal M}f(\mathbf{x})\|_F^2d\sigma.
\]
The discretization of the continuous Dirichlet energy is given immediately via the  approximation 
\[
    \frac{1}{2}\int_{{\cal M}_{ijk}} \|\nabla_{\cal M}f(\mathbf{x})\|_F^2\text{d}\sigma
    \approx \frac{1}{2}\int_{{\cal M}_{ijk}}\|f_{ijk}b_{ijk}^\top\|_F^2\text{d}\sigma.
\]
That is, the continuous Dirichlet energy is approximated by the discrete energy
\begin{align}\label{E_D}
    {\cal E}_D^h(f) 
    = \sum_{{\cal M}_{ijk}} 
    \frac{1}{2}\int_{{\cal M}_{ijk}}\|f_{ijk}b_{ijk}^\top\|_F^2\text{d}\sigma.
\end{align}
Using $b_{ijk}^\top = \frac{s_{ijk}^\top}{2A_{ijk}}$, the discrete Dirichlet energy can be represented as
\begin{align}\label{discr conformal energy}
    {\cal E}_D^h(f) 
    = \frac {1}{2}\sum_{{\cal M}_{ijk}} 
        \frac{{\cal A}_{ijk}}{2A_{ijk}}\|f_{ijk}s_{ijk}^\top\|_F^2
    = \frac {1}{2}\sum_{ij: e_{ij}\in\mathcal{E}(M)} 
        w_{ij}^* \|f(v_i)-f(v_j)\|^2
    = \frac{1}{2}\langle L\mathbf{f},\mathbf{f}\rangle,
\end{align}
where ${\cal A}_{ijk}$ is the area of ${\cal M}_{ijk}$, $\mathbf{f}$ is the matrix of rows $\{f_i^\top\}$, $L = (\ell_{ij})$ is the Laplacian matrix with $\ell_{ij} = -w_{ij}^*$ if $e_{ij}\in \mathcal{E}(M)$ or $\ell_{ij}=0$ if $e_{ij}\notin \mathcal{E}(M)$ for $i\neq j$, and $\ell_{ii} = \sum_{j\neq i}w_{ij}^*$, \begin{align}\label{def:omega}
    w_{ij}^* = \left\{\begin{array}{ll}
        \frac{\rho_{ij}\cot\beta_{ij}+\rho_{ji}\cot\beta_{ji}}{2}, & [v_i,v_j]\in \mathcal{E}(M\backslash\partial M);\\
        \frac{\rho_{ij}\cot\beta_{ij}}{2}, & [v_i,v_j]\in \mathcal{E}(\partial M),
    \end{array}\right.\quad
    \rho_{ij} = \frac{\mathcal{A}_{ijk}}{A_{ijk}},
\end{align}
and $\beta_{ij}$ is the angle opposite to the edge connecting $v_i$ and $v_j$ in $V_{ijk}$. 

{\bf Remark}. Since the ratio $\rho_{ij}\approx 1$, the discrete Dirichlet energy ${\cal E}_D^h(f)$ is a slight modification of that given in \cite{XGST20book} using $\rho_{ij}=1$. It is not easy to compute the areas ${\cal A}_{ijk}$. However, an estimation of each ${\cal A}_{ijk}$ is helpful to reduce the folding risk if the local curvature varies much.   

By Theorem \ref{thm:error esti} and the estimations by these lemmas given in the previous sections, we can bound the error of discrete Dirichlet energy to the continuous Dirichlet energy. 

\begin{theorem}\label{thm:discrete error}
If $\nabla_{\cal M}f$ is Lipschitz-continuous on $\cal M$, then
\begin{align}\label{epsilon}
    \big|\mathcal{E}_D(f)-\mathcal{E}_D^h(f)\big|
    \leq \varepsilon_D^h(f) 
    = \frac{1}{2}
    \int_{\mathcal{M}}\psi_f^2(\mathbf{x})\text{d}\sigma
    +\sqrt{2\mathcal{E}_D^h(f)
    \int_{\mathcal{M}}\psi_f^2(\mathbf{x})\text{d}\sigma}.
\end{align}
\end{theorem}
\begin{proof}
By Theorem \ref{thm:error esti},
\begin{align*}
    \big|\mathcal{E}_D(f)-{\cal E}_D^h(f)\big|
    &\leq\frac{1}{2}\sum
    \int_{{\cal M}_{ijk}}\Big|\|\nabla_{\!\cal M} f(\mathbf{x})\|_F^2
    -\big\|f_{ijk}b_{ijk}^\top\big\|_F^2\Big|\text{d}\sigma\\
    &\leq \frac{1}{2}\sum\int_{{\cal M}_{ijk}}\Big(\psi_f^2(\mathbf{x})
    +2\|f_{ijk}b_{ijk}^\top\|_F\psi_f(\mathbf{x})\Big)\text{d}\sigma\\
    &\leq \frac{1}{2}\int_{\mathcal{M}}\psi_f^2(\mathbf{x})\text{d}\sigma
    +\sqrt{\sum\int_{{\cal M}_{ijk}}
    \|f_{ijk}b_{ijk}^\top\|_F^2\text{d}\sigma
    \int_{\mathcal{M}}\psi_f^2(\mathbf{x})\text{d}\sigma}.
\end{align*}
Using the equality (\ref{E_D}), we get (\ref{epsilon}) immediately. 
\end{proof}

The integration $\int_{\mathcal{M}}\psi_f^2(\mathbf{x})\text{d}\sigma$ can be estimated based under some conditions, as shown in (\ref{bound:int psi}).

\section{Convergence of discrete conformal transformation}

Now we are ready to prove the convergence of discrete conformal energy. 
Besides the Lipschitz conditions mentioned before, the convergence conditions also depend on the triangulations on the surface $\cal M$ and the parameter domain $\Omega$. However, these two triangulations are equivalent to each other. Practically, 
\begin{align}\label{equiv:AT}
    \frac{ d(V_{ijk})}{\sigma_{\max}(\nabla \mathbf{x})}
    \leq d(\Omega_{ijk})\leq\frac{ d(V_{ijk})}{\sigma_{\min}(\nabla \mathbf{x})},\quad
    \frac{A_{ijk}}{T_{ijk}}\to \sqrt{\det(\nabla\mathbf{x}^\top\nabla\mathbf{x})}
\end{align}
as $d(V_{ijk})$ or $d(\Omega_{ijk})$ tends to zero.

Let $\sigma_{\max} = \max_{\omega\in\Omega} \|\nabla\mathbf{x}(\omega)\|_F$ and $\sigma_{\min} = \min_{\omega\in\Omega} \sigma_{\min}(\nabla\mathbf{x}(\omega))$, and consider the set
\[
    {\cal F} = \big\{f: \ f({\cal M})  
    = {\cal N}, c_L(f)\leq C_L\big\}.
\]
with a positive constant $C_L$.
By (\ref{bound:psi_f}), if $f$ belongs to the set $\cal F$, then for $\mathbf{x}\in{\cal M}_{ijk}$,
\[
	\epsilon_f
	=  2C_{\cal M}\| \nabla_{\cal M}f(\mathbf{x})\|_Fd^2(\Omega_{ijk})
	+C_L\sigma_{\max}^2d^2(\Omega_{ijk}) .
\]
Substituting it into (\ref{bound:psi_f}), we have that for $\mathbf{x}\in{\cal M}_{ijk}$,
\begin{align}\label{bound:psi2}
	\psi_f(\mathbf{x})
	&\leq  \| \nabla_{\cal M}f(\mathbf{x})\|_F \psi_{ijk}'  +   \psi_{ijk}'',
\end{align}
where  
\begin{align}\label{psi_{ijk}}
	\psi_{ijk}' = \frac{3C_{\cal M}d(V_{ijk})d^2(\Omega_{ijk})}{A_{ijk}}
	+\Big(\frac{3C_{\cal M}d^3(\Omega_{ijk})}{\sigma_{\min}T_{ijk}}\Big)^2,\quad
	\psi_{ijk}''  =  \frac{3C_L\sigma_{\max}^2d(V_{ijk})d^2(\Omega_{ijk})}{2A_{ijk}}.
\end{align}
Let $\psi'_{\max} =  \max \psi_{ijk}' $ and $\psi''_{\max} =  \max \psi_{ijk}''$, we see that
\begin{align}\label{bound:int psi}
	\frac{1}{2}\int_{\cal M}\psi_f^2(\mathbf{x})\text{d}\sigma
	&\leq {\cal E}_D(f)(\psi'_{\max})^2
	+\sqrt{2{\cal A}{\cal E}_D(f)}\psi'_{\max}\psi''_{\max}
	+\frac{1}{2}{\cal A}(\psi''_{\max})^2\nonumber\\
	&= \Big(\sqrt{{\cal E}_D(f)}\psi'_{\max}
	+\sqrt{\frac{1}{2}{\cal A}}\psi''_{\max}\Big)^2,
\end{align}
where $\cal A$ is the area of $\cal M$. Notice that both $\psi'_{\max}$ and $\psi''_{\max}$ also depend on the mesh size $h_\Omega$ or $h_V$,
\[
    h_{\Omega} = \max_{\Omega_{ijk}}d(\Omega_{ij}),\quad
    h_V = \max_{V_{ijk}}d(V_{ij}).
\]
For simplicity, by $h\to 0$, we mean $h_{\Omega}\to 0$ or $h_V\to 0$, equivalently. To highlight the dependency on the mesh, we will use the notation $\psi_{ijk}^{'h}$, $\psi_{ijk}^{''h}$ and  $\psi_f^h$ for the same $\psi_{ijk}'$, $\psi_{ijk}''$ and $\psi_f$, respectively.

\begin{theorem}\label{thm:convergence}
Suppose that $\{\psi_{ijk}^{'h}\}$ and $\{\psi_{ijk}^{''h}\}$ tend to zero uniformly as $h\to0$. Then
\begin{align}
    &\inf_{f\in {\cal F}} \mathcal{E}_D(f)
    = \lim_{h\to 0}\min_{f\in {\cal F}} \mathcal{E}_D^h(f).\label{equivalence}
\end{align}
\end{theorem}
\begin{proof}
On one hand, by Theorem \ref{thm:discrete error}, 
\[
	\mathcal{E}_D^h(f)
    	\leq \mathcal{E}_D(f) +\frac{1}{2}\int_{\cal M}\psi_f^2(\mathbf{x})\text{d}\sigma
        +2\sqrt{\frac{1}{2}\int_{\cal M}\psi_f^2(\mathbf{x})\text{d}\sigma\mathcal{E}_D^h(f)}.
\]
That is, $\Big(\sqrt{\mathcal{E}_D^h(f)}-\sqrt{\frac{1}{2}\int_{\cal M}\psi_f^2(\mathbf{x})\text{d}\sigma}\Big)^2
\leq \mathcal{E}_D(f)+\int_{\cal M}\psi_f^2(\mathbf{x})\text{d}\sigma$. Hence, for a fixed $f\in \cal F$,
\begin{align}\label{bound:{E}_D^h}
    \inf_{f\in {\cal F}} \mathcal{E}_D^h(f)
    \leq\mathcal{E}_D^h(f)
    \leq \Big(\sqrt{\frac{1}{2}\int_{\cal M}\psi_f^2(\mathbf{x})\text{d}\sigma}
        + \sqrt{\mathcal{E}_D(f)+\int_{\cal M}\psi_f^2(\mathbf{x})\text{d}\sigma}\Big)^2.
\end{align}
By the uniform descending of $\{\psi_{ijk}^{'h}\}$ and $\{\psi_{ijk}^{''h}\}$, 
$\lim_{h\to 0}\inf_{f\in {\cal F}} \mathcal{E}_D^h(f)
    \leq \mathcal{E}_D(f)$. Hence,
\begin{align}\label{discr:upper}
    \lim_{h\to 0}\inf_{f\in {\cal F}} \mathcal{E}_D^h(f)
    \leq \inf_{f\in {\cal F}}\mathcal{E}_D(f).
\end{align}

On the other hand, by Theorem \ref{thm:discrete error}, we also have that for a fixed $f\in {\cal F}$,
\begin{align*}
	\mathcal{E}_D(f) 
	\leq \mathcal{E}_D^h(f) 
	+\frac{1}{2}\!\int_{\cal M}\!\psi_f^2(\mathbf{x})\text{d}\sigma
        +2\sqrt{\frac{1}{2}\!\int_{\cal M}\! \psi_f^2(\mathbf{x})\text{d}\sigma\mathcal{E}_D^h(f)}
    = \left\{\!\sqrt{\mathcal{E}_D^h(f)}
        +\sqrt{\frac{1}{2}\!\int_{\cal M}\! 
        \psi_f^2(\mathbf{x})\text{d}\sigma}\right\}^2\!.
\end{align*}
Hence, combining it with (\ref{bound:int psi}), we get
\begin{align*}
	\sqrt{\mathcal{E}_D^h(f)}
	\geq \sqrt{\mathcal{E}_D(f)}-\sqrt{\frac{1}{2}\int_{\cal M}\psi_f^2(\mathbf{x})\text{d}\sigma}
	\geq \sqrt{\mathcal{E}_D(f)}(1-\psi'_{\max})-\sqrt{\frac{1}{2}{\cal A}}\psi''_{\max}.
\end{align*}
That is,  $\sqrt{\mathcal{E}_D(f)}(1-\psi'_{\max})
	\leq \sqrt{\mathcal{E}_D^h(f)}+\sqrt{\frac{1}{2}{\cal A}}\psi''_{\max}$. It follows that
\[
	\inf_{f\in{\cal F}}\mathcal{E}_D(f)(1-\psi'_{\max})^2
	\leq \mathcal{E}_D(f)(1-\psi'_{\max})^2
	\leq \big(\sqrt{\mathcal{E}_D^h(f)}+\sqrt{\frac{1}{2}{\cal A}}\psi''_{\max}\big)^2.
\]
Letting $h \to 0$, we get $\inf_{f\in{\cal F}}\mathcal{E}_D(f)\leq \lim_{h\to 0}\inf_{f\in{\cal F}}\mathcal{E}_D^h(f)$. Combining it with (\ref{discr:upper}), the theorem is proven. 
\end{proof}

The uniformity assumption in Theorem \ref{thm:convergence} holds under a weak condition. 
To divide the condition, let $\theta(\Omega_{ijk})$ and $\theta(V_{ijk})$ be the smallest angle of $\Omega_{ijk}$ and $V_{ijk}$, respectively. Since 
\begin{align}\label{equiv:d}
	\frac{1}{4}d^2(\Omega_{ijk}) \leq \frac{T_{ijk}}{\sin\theta(\Omega_{ijk})}\leq \frac{1}{2}d^2(\Omega_{ijk}),\quad
	\frac{1}{4}d^2(V_{ijk}) \leq \frac{A_{ijk}}{\sin\theta(V_{ijk})}\leq \frac{1}{2}d^2(V_{ijk}),
\end{align}
and by the equivalence (\ref{equiv:AT}), both the two terms $\psi_{ijk}'$  and $\psi_{ijk}''$ in (\ref{psi_{ijk}}) can be further bounded as
\begin{align}\label{psi_{ijk}'}
	\psi_{ijk}' \leq \frac{12C_{\cal M}}{\sigma_{\min}^2}\frac{d(V_{ijk})}{\sin\theta(V_{ijk})}
	+\Big(\frac{12C_{\cal M}}{\sigma_{\min}}\frac{d(\Omega_{ijk})}{\sin\theta(\Omega_{ijk})}\Big)^2,\quad
	\psi_{ijk}''  \leq \frac{6C_L\sigma_{\max}^2}{\sigma_{\min}^2}\frac{d(V_{ijk})}{\sin\theta(V_{ijk})}. 
\end{align}
The ratios $\frac{d(V_{ijk})}{\sin\theta(V_{ijk})}$ and $\frac{d(\Omega_{ijk})}{\sin\theta(\Omega_{ijk})}$
are also equivalent when they tend to zero by (\ref{equiv:AT}) and (\ref{equiv:d}).
Therefore,  by the  equivalence and (\ref{psi_{ijk}'}), we obtain the following theorem.

\begin{theorem}\label{thm:conv_cond}
Suppose that ${\cal M} = \big\{\mathbf{x}(\omega):\ \omega\in\Omega\big\}$, $\nabla\mathbf{x}(\omega)$ is Lipschitz-continuous and not rank-deficient on $\Omega\subset{\cal R}^2$. Then the equivalence (\ref{equivalence}) holds if the triangle partition $\{V_{ijk}\}$ of $\cal M$ satisfy
\begin{align}\label{conv:cond}
    \max_{ijk}\frac{d(V_{ijk})}{\sin\theta(V_{ijk})}\to 0\quad
    \textrm{as}\quad h_V\to 0.
\end{align}
\end{theorem}

The condition (\ref{conv:cond}) asks for a not very poor triangulation, and is satisfied generally, for instance, if the smallest angles $\theta(V_{ijk})$ have a positive lower bound. The commonly used Delaunay triangulation of a given set of vertices can increase the possibility of satisfying the condition since Delaunay triangulation maximizes the minimum angle in the triangulation. 
In \cite{Hutc91}, the quasi-uniform condition 
\begin{align}\label{quasi-uniform}
    \frac{d(V_{ijk})}{r(V_{ijk})} \leq C,
\end{align}
is used to guarantee the convergence of discrete conformal mapping between two 2D sets, where $r(V_{ijk})$ is radius of inscribed circle of $V_{ijk}$. By $r(V_{ijk})=\frac{2A_{ijk}}{\|v_{ij}\|+\|v_{jk}\|+\|v_{ki}\|} \leq \frac{1}{2}d(V_{ijk})\sin\theta(V_{ijk})$, 
\[
    \frac{d(V_{ijk})}{\sin\theta(V_{ijk})}
    =\frac{d(V_{ijk})}{r(V_{ijk})}
    \frac{r(V_{ijk})}{\sin\theta(V_{ijk})}
    \leq \frac{C}{2}d(V_{ijk}),
\]
the condition (\ref{conv:cond}) is satisfied obviously under the quasi-uniform condition (\ref{quasi-uniform}).

Finally, we show a spacial proposition of Delaunay triangulation when $\min\theta(V_{ijk}^{(t)})\to 0$. 
It is known that Delaunay triangulation cannot avoid the degradation of minimal angle tending to zero as $h\to 0$. The following lemma shows that when the degradation happens, there are a sequence of the triangles with two approximately equal edges and the smallest edge is a higher order of them.  

\begin{lemma}\label{lma:edge}
Given a sequence of Delaunay triangulation $\{V_{ijk}^{(t)}\}$, if
$\min\theta(V_{ijk}^{(t)})\to 0$ as $\max d(V_{ijk}^{(t)})\to 0$, then there is a subsequence $\{V_{ijk}^{(t')}\}$ with $\|v_{ij}^{(t')}\|_2\leq \|v_{jk}^{(t')}\|_2\leq \|v_{ki}^{(t')}\|_2$ such that 
\begin{align}\label{conv:edge}
    \frac{\|v_{ij}^{(t')}\|_2}{ \|v_{ki}^{(t')}\|_2}\to 0,\quad \frac{\|v_{jk}^{(t')}\|_2}{\|v_{ki}^{(t')}\|_2}\to 1. 
\end{align}
\end{lemma}
\begin{proof}
Let $V_{ijk}^{(t)} = \arg\min_{ijk}\theta(V_{ijk}^{(t)})$ for each $t$. For  simplicity, we also use $V_{ijk}$ as $V_{ijk}^{(t)}$, and let $a = \|v_{ij}\|_2$, $b =\|v_{jk}\|_2$, $c = \|v_{ki}\|_2$, and assume $a\leq b\leq c$ with the opposite angles $0<\theta_a\leq\theta_b\leq\theta_c$, respectively. 
Furthermore, let $V_{i'j'k'} = V_{j'ik}$ with $v_{i'}=v_k$ and $v_{k'} = v_i$, sharing the same edge with $V_{ijk}$, and  
\[
    a' = \|v_{i'j'}\|_2 = \|v_{kj'}\|_2, \quad
    b' = \|v_{j'k'}\|_2 = \|v_{j'i}\|_2, \quad
    c' = \|v_{k'i'}\|_2 = \|v_{ki}\|_2 = c.
\]
Consider the two cases $\theta_b\geq \theta_0>0$ and $\theta_b\to 0$.

(1) If $\theta_b\geq \theta_0>0$ with a constant $\theta_0<\pi/2$, then 
$\theta_c = \pi-\theta_a-\theta_b\leq\pi-\theta_a-\theta_0$. Then
\begin{align*}
    &1\geq\frac{b}{c} = \frac{\sin\theta_b}{\sin\theta_c}
    = \frac{\sin\theta_b}{\sin(\theta_a+\theta_b)}
    \geq \frac{\sin\theta_0}{\sin(\theta_a+\theta_0)}\to 1,\\
    &0<\frac{a}{c} = \frac{\sin\theta_a}{\sin(\theta_a+\theta_b)}
    \leq \frac{\sin\theta_a}{\sin(\theta_a+\theta_0)}\to 0.
\end{align*}

(2) If $\theta_b\to 0$, which implies that $\theta_c\to \pi$, we must have $\theta_{c'}\to 0$ and $\sin\theta_{c'}\leq \sin(\pi-\theta_c) =  \sin\theta_c$ since $\theta_c+\theta_{c'}<\pi$ as Delaunay triangles. Thus, if both of $\theta_{a'}$ and $\theta_{b'}$ do not tend to zero or $\pi$, by the law of sines, we have
\[
    \frac{c}{a'}
    =\frac{\sin\theta_{c'}}{\sin\theta_{a'}} \to 0,\quad
    \frac{c}{b'}
    =\frac{\sin\theta_{c'}}{\sin\theta_{b'}} \to 0.
\]
By the law of cosines, 
\[
    \frac{a'^2}{b'^2} = \frac{b'^2+c^2-2b'c\cos\theta_a}{b'^2} \to 1.
\]
Therefore, we always have a subsequence of the triangles satisfying (\ref{conv:edge}).
\end{proof}

This proposition (\ref{conv:edge}) helps us to further improve the Delaunay triangulation if necessary. We may modify the vertices $\{v_i\}$ to avoid (\ref{conv:edge}) if it happens. When it is done, a Delaunay triangulation  has a positive lower bound for the angles and guarantees $\max_{ijk}\frac{d(V_{ijk})}{\sin\theta(V_{ijk})}\to 0$.


\begin{figure}[t]
    \centering
    \begin{tabular}{ccc}
        \includegraphics[height = 3.5cm]{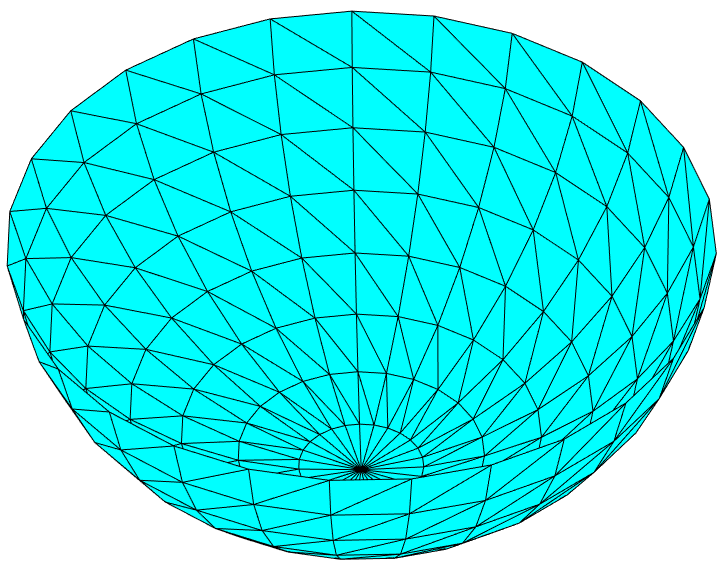}\hspace{10pt} & 
        \includegraphics[height = 3.5cm]{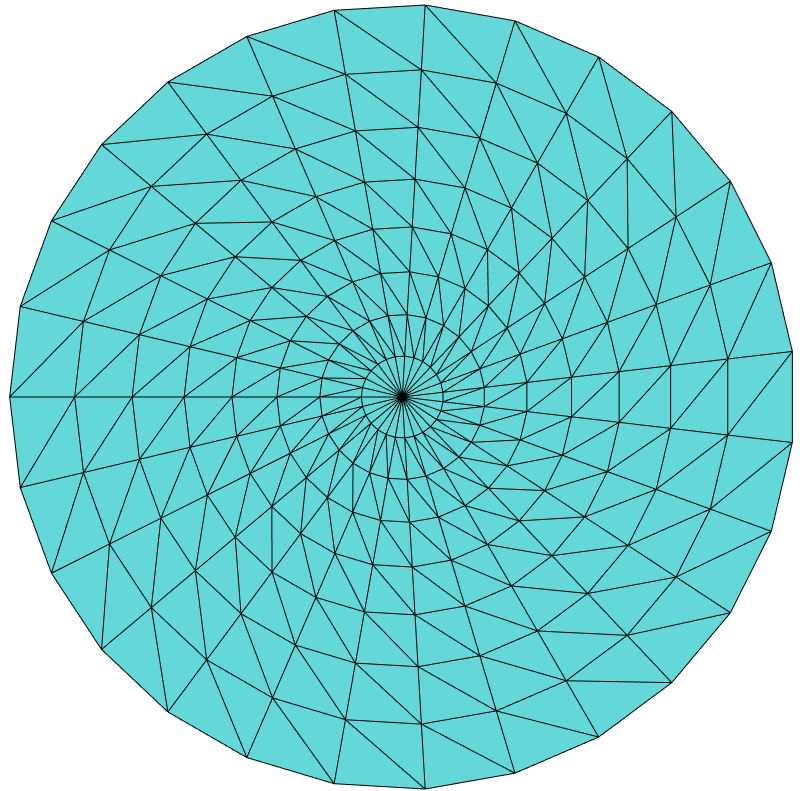} &
        \includegraphics[height = 3.5cm]{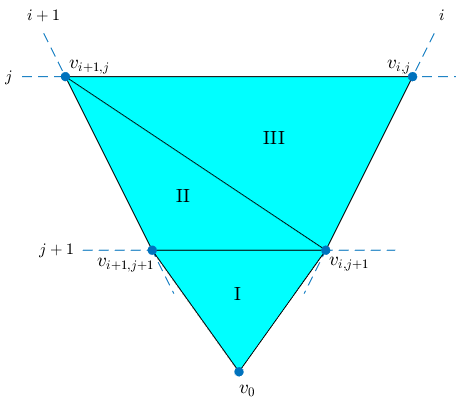}\\
        (a) & (b) & (c)
    \end{tabular}
    \caption{The triangulation of hemisphere (a) and its stereographic projection on to the disk (b) with $m = 27$, $n = 8$. There are tree types of triangles in the triangulation (c).}
    \label{fig:disc_hmsdisk}
\end{figure}

\section{Numerical Experiment}

The conditions of convergence given in the analysis is sufficient theoretically, especially the condition (\ref{conv:cond}) on the triangulation. In this section, we show that this condition is very tight numerically. We consider the stereographic projection that transforms the unit south hemisphere into a unit disk conformally.  
We use the sphere coordinate to parameterize the hemisphere,
\[
    \mathbf{x} = 
         (\cos\phi\sin\psi,\ \sin\phi\sin\psi,\ \cos\psi)^\top,\quad
         \phi\in[0,2\pi],\quad
         \psi\in[\frac{\pi}{2},\pi],
\]
and take $m$ points in the interval $[0,2\pi]$ and $n$ points in $[\frac{\pi}{2},\pi]$ in equal distance as
\begin{align*}
    &\phi_i = \frac{2i\pi}{m},\quad
    \psi_j = \frac{j\pi}{2n} + \frac{\pi}{2}, \quad 
    i = 0,1,\cdots,m-1, \ j = 0,1,\cdots,n-1.
\end{align*}
Hence, we have totally $mn+1$ vertices
$\{v_{ij} = (\cos\phi_i\sin\psi_j,\ \sin\phi_i\sin\psi_j,\ \cos\psi_j)\}$ and the south pole $v_0=(0,0,-1)^\top$ on the surface.
The edges are chosen as follows to generate a triangulation. 
\begin{itemize}
    \item $[v_{ij}, v_{i,j+1}]$, $[v_{ij}, v_{i,j+1}]$ and $[v_{i+1,j}, v_{i,j+1}]$ with $v_{m,j} = v_{0,j}$, for $i\leq m-1$ and $j\leq n-2$;
    \item $[v_{i,n-1}, v_0]$ for $i = 0,1,2,\cdots,m-1$.
\end{itemize}
Figure \ref{fig:disc_hmsdisk} illustrates the triangulation and its stereographic projection onto the unit disk.

Clearly, there are three types of triangles in the triangulation as shown in (c) of Figure \ref{fig:disc_hmsdisk}:
\begin{itemize}
    \item[] Type I. The triangles with the south pole $V_i = [v_0,v_{i,n-1},v_{i+1,n-1}]$ with the edge lengths
    \[
        \|v_0-v_{i,n-1}\|=\|v_0-v_{i+1,n-1}\|=2\sin\frac{\pi}{4n},\quad
        \|v_{i+1,n-1}-v_{i,n-1}\| = 2\sin\frac{\pi}{m}\cos\frac{(n-1)\pi}{2n}.
    \]
    \item[] Type II. The triangles $V_{ij}' = [v_{i+1,j},v_{i+1,j+1},v_{i,j+1}]$ with the largest edge length
    \[
         \|v_{i+1,j}-v_{i,j+1}\|
         = \sqrt{4\sin^2\frac{\pi}{4n} + 4\cos\frac{(j+1)\pi}{2n}\cos\frac{j\pi}{2n}\sin^2\frac{\pi}{m}}.
    \]
    \item[] Type III. The triangles $V_{ij}'' = [v_{i+1,j},v_{i,j+1},v_{i,j}]$ for $i\leq m-1$ and $j\leq n-2$ with the edge lengths
    \[
        \|v_{i+1,j}-v_{i,j+1}\| > \|v_{i,j}-v_{i,j+1}\| = 2\sin\frac{\pi}{4n}, \quad
        \|v_{i+1,j}-v_{i,j}\| = 2\sin\frac{\pi}{m}\cos\frac{j\pi}{2n}.
    \]
\end{itemize}

\begin{figure}[t]
    \centering
    \begin{tabular}{cc}
         \includegraphics[width = 7cm]{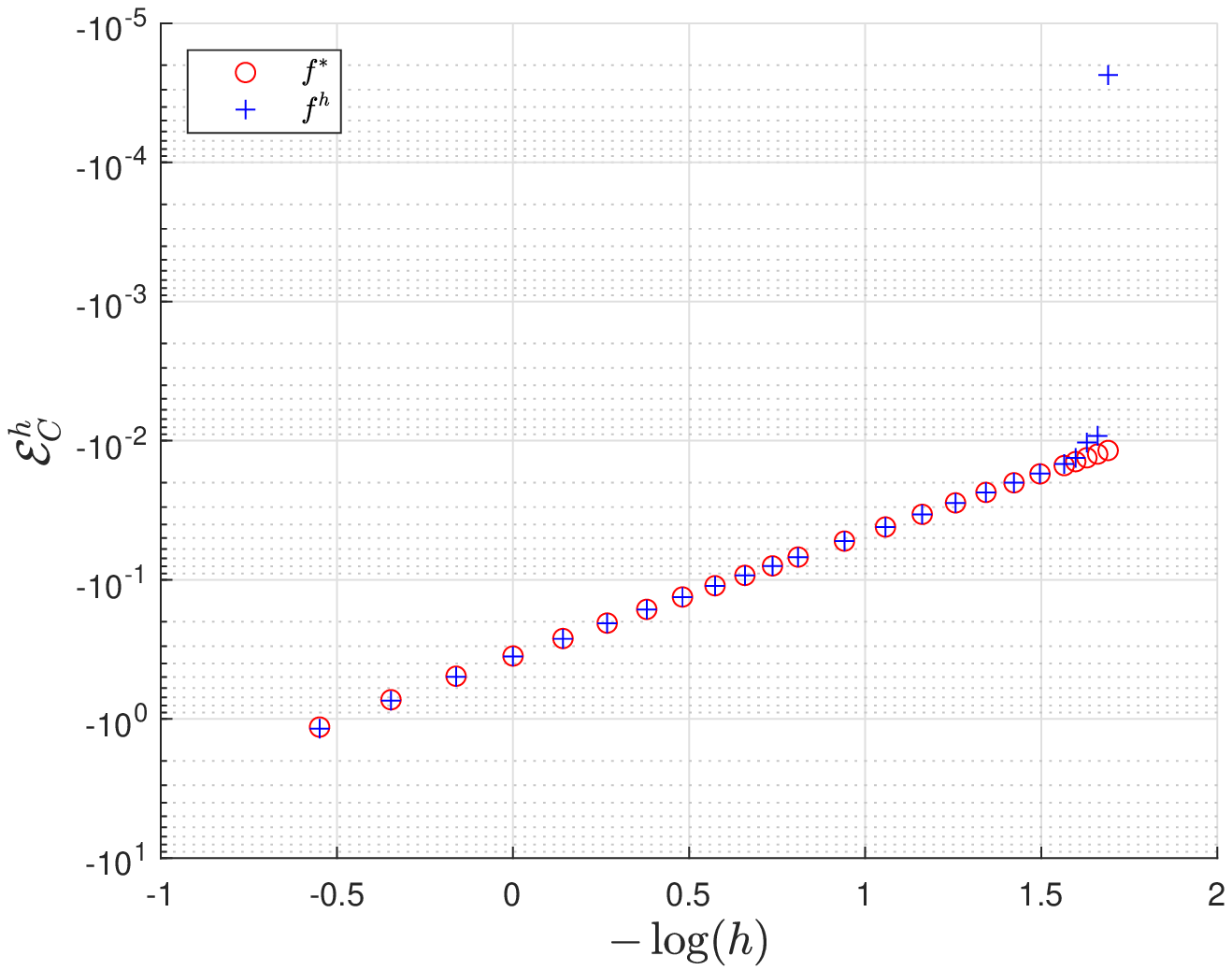} &
         \includegraphics[width = 7cm]{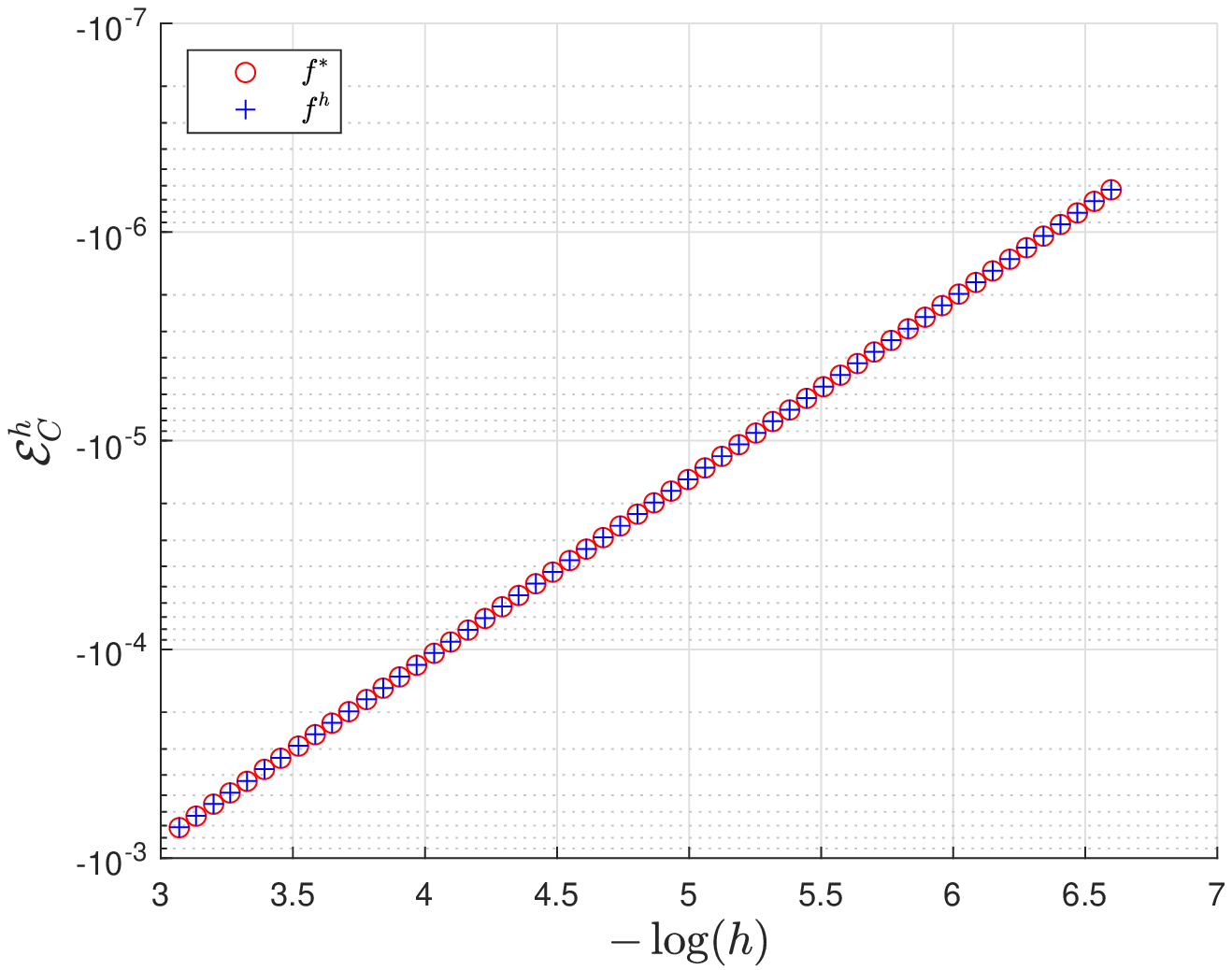}
    \end{tabular}
    \caption{Discrete conformal energy at $r = \frac{1}{4}$ (left) and $r = \frac{11}{12}$ (right).}
    \label{fig:comformal energy}
\end{figure}

The condition (\ref{conv:cond}) depends on a complicated relation between $m$ and $n$ when they tend to infinity. In this experiment, we choose $m=\lfloor n^r \rfloor$ with $r=1/4$ or $r=11/12$, in which the condition (\ref{conv:cond}) can be easily checked as $n\to \infty$ since in the two cases, we always have 
\[
    \lim_{n\to\infty}\frac{d(V_i)}{\sin\theta(V_i)}
        \leq \lim_{n\to\infty}\frac{d(V_{ij}'')}{\sin\theta(V_{ij}'')}
        =\lim_{n\to\infty}\frac{d(V_{ij}')}{\sin\theta(V_{ij}')}.
\]
However, the limit of $\frac{d(V_{ij}')}{\sin\theta(V_{ij}')}$ is zero when $r=11/12$, and it is not zero when $r=1/4$. Hence, the discrete solution $f^h$ that minimizing the discrete conformal energy can converge to the ideal conformal mapping $f^*$ if $r=11/12$ theoretically, but $f^h$ may fail to converge to $f^*$ if $r=1/4$. Numerical computation also supports the conclusion.

We minimize the discrete Dirichlet energy ${\cal E}_D(f)$ using the algorithm CEM given in \cite{MHWW17} on both the triangulation settings, and check the behaviour of the minimal discrete conformal energy ${\cal E}_C(f^h)$ of the solution $f^h$ and the discrete conformal energy ${\cal E}_C(f^*)$ of the ideal conformal mapping $f^*$. We also check the relative error of $f^h$ with respect to $f^*$, defined as the total errors at all the vertices $V$
\[
    \varepsilon(f^h) = \frac{\|f^h(V)-f^*(V)\|_F}{\|f^*(V)\|_F}.
\]

Both the discrete conformal energies ${\cal E}_C(f^h)$ and ${\cal E}_C(f^*)$ are negative and tend to zero as the mesh size $h\to 0$, or $n\to \infty$ equivalently. For $r = 1/4$, they are matched to each other if $h>0.2091$. However, if $h$ decreases, the discrete solution $f^h$ may be over fitting -- ${\cal E}_C(f^h)$ is much smaller than ${\cal E}_C(f^*)$ in magnitude. However, for $r = 11/12$, ${\cal E}_C(f^h)$ matches ${\cal E}_C(f^*)$ very well. See Figure \ref{fig:comformal energy} for the dependence of the discrete conformal energy on the mesh size $h$ for both of $f^h$ and $f^*$. This phenomenon supports the importance of the condition (\ref{conv:cond}) for the convergence of discrete solution.

It is a bit interesting that the approximation error $\varepsilon(f^h)$ of $f^h$ depends on $h^c$ with a constant $c\approx 1.0890$ linearly. Figure \ref{fig:conv_r} shows the phenomenon. The linear dependency is obeyed for $r=1/4$ when $h$ is not sufficiently small. When $h$ decreases and is smaller than $0.2322$, the relative error ascends quickly. However, for $r=11/12$, the relative error descends continuously and the linear dependency is still preserved as $h$ tends to zero. We believe that the excellent performance dues to the convergence condition (\ref{conv:cond}).

\begin{figure}[t]
    \centering
    \begin{tabular}{cc}
         \includegraphics[width = 7cm]{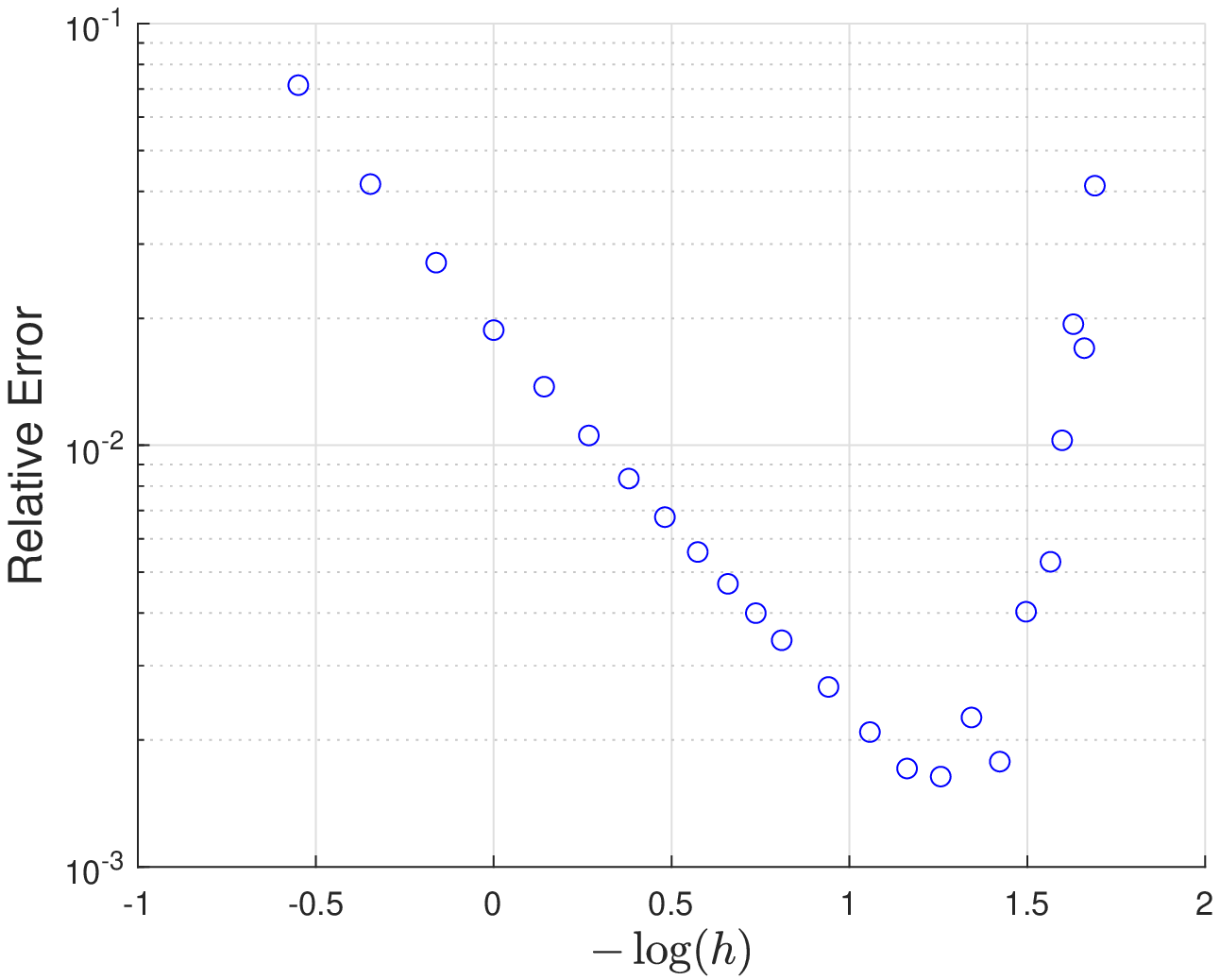} &
         \includegraphics[width = 7cm]{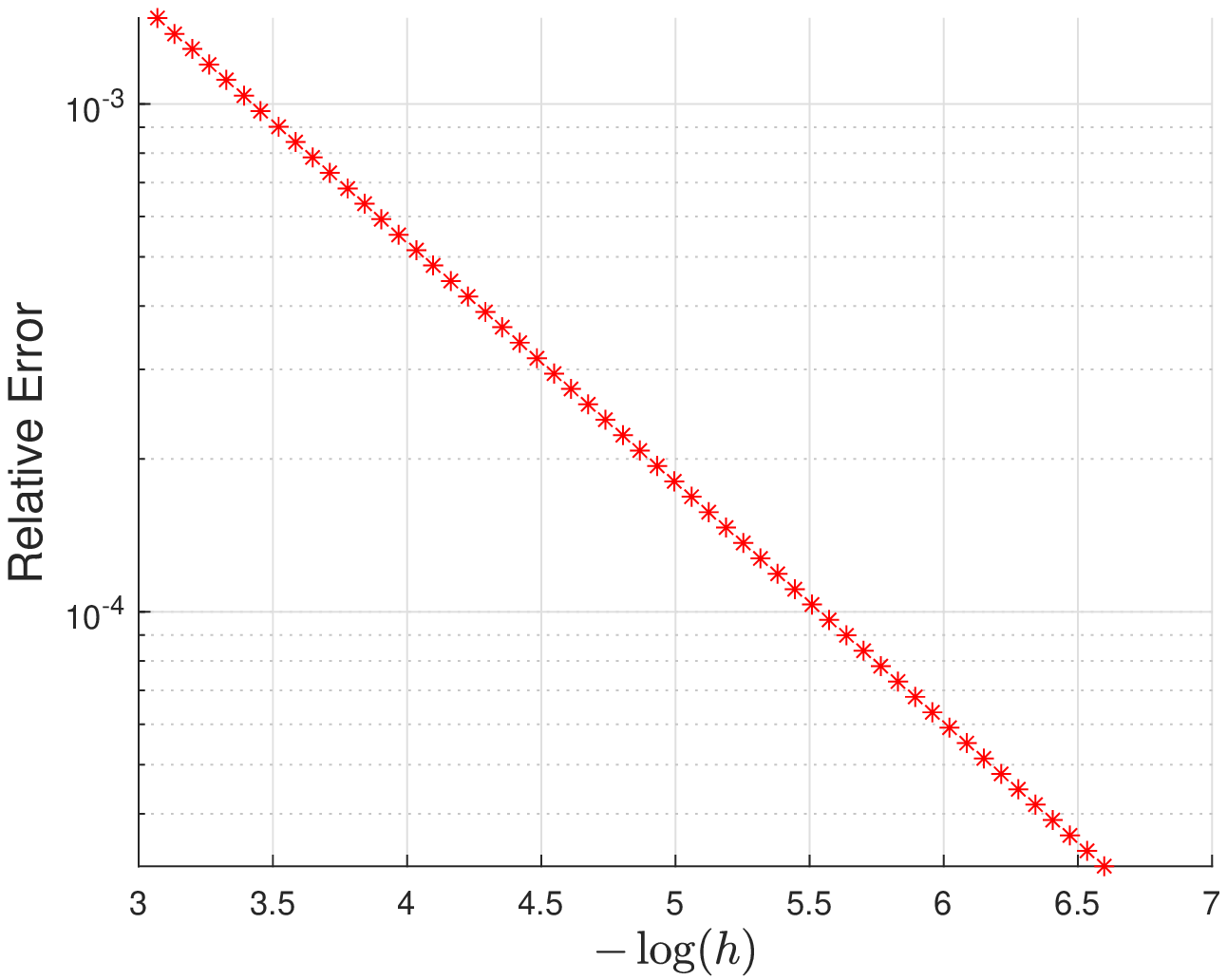}
    \end{tabular}
    \caption{Convergence performance at $r = \frac{1}{4}$ (left) and $r = \frac{11}{12}$ (right).}
    \label{fig:conv_r}
\end{figure}

\section{Conclusions}

In this paper, we addressed the convergence of minimizing discrete  conformal energy for finding the conformal mapping between two predicted surfaces $\cal M$ and $\cal N$ if the conformal mapping exists. The sufficient condition for the convergence on the triangulation is weak on one side. On the other side, it also predicts the risk of a discrete solution far from the true conformal mapping in the given vertices when the condition is not satisfied since error may be incredibly large. This issue can be addressed via modifying the triangulation. 

However, in numerical computation, there are two issues should be taken into account fully. 
Theoretically, the considered continuous mapping from ${\cal M}$ onto ${\cal N}$ should be bijective obviously. This one-to-one and on-to properties should be also preserved in the discrete model (\ref{discr conformal energy}). This implicit restriction, however, may be ignored or hard to be obeyed in a discrete solution generally. Practically, even if the triangulation is good, a discrete solution with a very small discrete conformal energy may be also far from the conformal mapping, when (1) the solution is degenerated because it does not preserve the on-to property or, (2) the solution multiply covers the target surface, which means it loses the one-to-one property. The one-to-one property may be also partially demolished in solving a discrete optimization problem transformed from the minimization of discrete conformal energy. For example, it is commonly used to determine interior points from boundary points via solving a linear system. In the case when the coefficient matrix is ill-conditioned, the solved interior points may be folded, which leads to a solution that is not one-to-one obviously.

The stable computation of minimizing discrete conformal energy is a worthy topic in this line. In our recent work in preparing \cite{}, we considered the stable computation for the disk parameterization of conformal mappings on open surfaces. The minimization problem of discrete conformal energy is modified, using some penalization strategies to address the three kinds of risks mentioned above. However, this modified model is hard to be transformed for closed surfaces. We will continue this work on closed surfaces.

\appendix

\section{Proof of conformal equivalence}\label{Appendix:CE}

Since $\mathcal{M}$ is parameterized by $\mathbf{x}=\mathbf{x}(u,v)$, the tangent space ${\cal T}_{\mathbf{x}}\cal M$ is spanned by $\nabla\mathbf{x}=[\mathbf{x}_u,\mathbf{x}_v]$.  
Similarly, the tangent space ${\cal T}_{\mathbf{y}}\cal M^*$ is spanned by $\nabla \mathbf{y}=[\mathbf{y}_u,\mathbf{y}_v]$ with $\mathbf{y}= f(\mathbf{x})$.
Hence, the representations 
\[
    x = \mathbf{x}_us+\mathbf{x}_vt
    \in {\cal T}_{\mathbf{x}}{\cal M},\quad
    y = \mathbf{y}_us+\mathbf{y}_vt
    \in{\cal T}_{\mathbf{x}}{\cal M^*},\quad
    \forall \ (s,t)\in  {\cal R}^2
\]
build a mapping between ${\cal T}_{\mathbf{x}}\cal M$ and ${\cal T}_{\mathbf{y}}\cal M^*$. That is,
\[
    y = \nabla\mathbf{y}[s, t]^\top
    = \nabla\mathbf{y}(\nabla\mathbf{x}^\top \nabla\mathbf{x})^{-1}\nabla\mathbf{x}^\top x
    = (\nabla_{\cal M}f )x.
\]
Assigning the orthogonal basis $Q$ and $G$ for the two tangent spaces, 
respectively, and using $x =  QQ^\top x$,
the coordinates $Q^\top x$ and $G^\top y$ satisfy the equation
\[
    G^\top y 
    = \big(G^\top \nabla_{\cal M}f Q\big)(Q^\top x)
    = J (Q^\top x),\quad
    J = G^\top \nabla_{\cal M}f Q.
\]
That is, $J$ is a linear transformation from ${\cal T}_{\mathbf{x}}\cal M$ to ${\cal T}_{\mathbf{y}}\cal M^*$. Hence, the area of $f(\mathcal{M}) = \mathcal{M}^*$ is
\[
    \mathcal{A}(f) = \int_{\mathcal{M}^*}\text{d}s_{\mathcal{M}^*}
    = \int_{\mathcal{M}}|\det(J)|\text{d}s_{\mathcal{M}}
    \leq \frac{1}{2}\int_{\mathcal{M}}\|J\|_F^2\text{d}\sigma
    \leq \frac{1}{2}\int_{\mathcal{M}}\|\nabla_{\!\cal M} f\|_F^2\text{d}\sigma
    =\mathcal{E}_D(f).
\]
Here we have used the inequality $|\det(A)|\leq \frac{1}{2}\|A\|_F^2$ for any $2\times2$ matrix, and $\|J\|_F\leq \|\nabla_{\!\cal M}f\|_F$.
Hence, $\mathcal{E}_C(f)=\mathcal{E}_D(f)-\mathcal{A}(f)\geq 0$. Clearly, $\mathcal{E}_C(f)=0$ if and only if $|\det(J)|=\frac{1}{2}\|J\|_F^2$ that means $J=\sigma H$ with a scale $\sigma\geq0$ and an orthogonal $H$, and $\|J\|_F=\|\nabla_{\!\cal M}f\|_F$ that means $\nabla_{\!\cal M}f = GJQ^\top = \sigma GHQ^\top$. Hence, 
\[
    \nabla f(\mathbf{x})^\top \nabla f(\mathbf{x})
    = \big(\nabla_{\!\cal M}f\nabla\mathbf{x}\big)^\top
    \big(\nabla_{\!\cal M}f\nabla\mathbf{x}\big)
    =\sigma^2\nabla\mathbf{x}^\top QQ^\top \nabla\mathbf{x}
    = \sigma^2\nabla\mathbf{x}^\top\nabla\mathbf{x}.
\]
That is, $f$ is conformal.

\section{Discretization of Dirichlet energy on $S_V$}\label{Appendix:DE S_V}

Let $S_V$ be triangle surface of ${\cal M}$ with given vertices $\{v_\ell\}$ and triangles $\{V_{ijk}\}$, and let $f$ be a piecewise linear mapping on $S_V$. That is, $f(v)$ is a linear mapping from $v\in V_{ijk}$ to ${\cal R}^n$. We can represent $f$ as
\[
    f(v) = P_{ijk} v+c_{ijk},\quad  \forall v\in V_{ijk},
\]
where $P_{ijk}$ is an $n\times m$ matrix, and $c_{ijk}\in {\cal R}^n$. By definition, in the triangle $V_{ijk}$, $\nabla_{S_V}f = P_{ijk}$.

Let $f_{ij} = f_i-f_j$ and $v_{ij} = v_i-v_j$. We have  $[f_{ij}, f_{jk}] = P_{ijk}[v_{ij},v_{jk}]$. It implies 
\begin{align}\label{P_ijk}
    P_{ijk}
    &= [f_{ij},f_{jk}]
    \big([v_{ij},v_{jk}]^\top[v_{ij},v_{jk}]\big)^{-1}[v_{ij},v_{jk}]^\top.
\end{align}
Using the equalities
\begin{align}
    &\big([v_{ij},v_{jk}]^\top[v_{ij},v_{jk}]\big)^{-1}
    = \frac{1}{(2A_{ijk})^2}[v_{jk},v_{ji}]^\top [v_{jk},v_{ji}], \label{eq:inv}\\
    &2A_{ijk} = \|v_{ki}\|\|v_{kj}\|\sin\beta_{ij}
    = \|v_{ik}\|\|v_{ij}\|\sin\beta_{jk} 
    = \|v_{jk}\|\|v_{ji}\|\sin\beta_{ki},\label{eq:A_ijk}
\end{align}
we get
\begin{align*}
    \|P_{ijk}\|_F^2 &=\langle P_{ijk},P_{ijk}\rangle
    =\big\langle [f_{ij}, f_{jk}]^\top[f_{ij}, f_{jk}],([v_{ij},v_{jk}]^\top[v_{ij},v_{jk}])^{-1}\big\rangle\\
    &=\frac{1}{4A_{ijk}^2}\big\langle [f_{ij},f_{jk}]^\top[f_{ij},f_{jk}], 
    [v_{jk},v_{ji}]^\top [v_{jk}, v_{ji}]\big\rangle
    = \frac{1}{4A_{ijk}^2}\|f_{ij}v_{jk}^\top+f_{jk}v_{ji}^\top\|_F^2\\
    &=\frac{1}{4A_{ijk}^2}\big(\|f_{ij}\|^2\|v_{jk}\|^2+\|f_{jk}\|^2\|v_{ji}\|^2
    +2f_{ij}^\top f_{jk} v_{jk}^\top v_{ji}\big).
\end{align*} 
By $f_{ki}=-f_{ij}-f_{jk}$,
$2f_{ij}^\top f_{jk}=\|f_{ki}\|^2 - \|f_{ij}\|^2-\|f_{jk}\|^2$. Substituting it into the above equality, we get
\begin{align*}
    \|P_{ijk}\|_F^2
    &=\frac{1}{4A_{ijk}^2}\big(\|f_{ij}\|^2(\|v_{jk}\|^2-v_{jk}^\top v_{ji})  
     +\|f_{jk}\|^2(\|v_{ji}\|^2-v_{jk}^\top v_{ji}) 
     + \|f_{ki}\|^2v_{jk}^\top v_{ji}\big)\\
    &=\frac{1}{4A_{ijk}^2}\big(\|f_{ij}\|^2v_{ki}^\top v_{kj} 
     +\|f_{jk}\|^2v_{ik}^\top v_{ij} +\|f_{ki}\|^2v_{jk}^\top v_{ji}\big)\\
    &=\frac{1}{2A_{ijk}}\big(\|f_{ij}\|^2\cot\beta_{ij}
     +\|f_{jk}\|^2\cot\beta_{jk} +\|f_{ki}\|^2\cot\beta_{ki}\big).
\end{align*}
Therefore, 
\begin{align*}
    \int_{S_V}\|\nabla_{S_V}f\|_F^2\text{d}\sigma
    &= \sum_{V_{ijk}}\|\nabla_{S_V}f\|_F^2A_{ijk} 
    = \sum_{V_{ijk}}\|P_{ijk}\|_F^2A_{ijk}\\
    &= \frac{1}{2}\sum_{V_{ijk}} 
    \big(
    \|f_{ij}\|^2\cot\beta_{ij}+\|f_{jk}\|^2\cot\beta_{jk} +\|f_{ki}\|^2\cot\beta_{ki}\big)\\
    &=\frac{1}{2} \sum_{ij}\omega_{ij}\|f_i-f_j\|^2.
\end{align*}

\section{Discrete weak solution of Laplace-Beltrami equation on $S_V$}\label{Appendix:weakLB}

The weak formulation \eqref{LB_weak} on $S_V$ is that for any $g$,
\begin{align}\label{LB_weakVS}
    \int_{S_V}\big\langle\nabla_{S_V}f,\nabla_{S_V}g\big\rangle\text{d}\sigma
    = \Big\langle \big(\frac{\partial}{\partial u}, 
        -\frac{\partial}{\partial v} \big),g \Big\rangle\Big|_p, 
\end{align}
where both $f$ and $g$ are piecewise linear on $S_V$, and $\nabla_{S_V}f = P_{ijk}$. For each vertex $v_i$, let ${\cal N}(i)$ be the set of triangles $\{V_{ijk}\}$  containing $v_i$ as one of its vertices. Consider a special piecewise linear function $g$ that maps $S_V$ to zero except these $\{V_{ijk}\}$, and for $v\in V_{ijk}$,
\[
    g(v) = G_{ijk}v+c_{ijk}, \quad g(v_i) = e_t,\ g(v_j)=g(v_k) = 0. 
\]
where $t=1$ or $t=2$, and $e_1=(1,0)^\top, e_2 = (0,1)^\top$, which gives $[e_t,0] = G_{ijk}[v_{ij},v_{jk}]$. Similar with the $P_{ijk}$, we also have
$G_{ijk}=[e_t,0]\big([v_{ij},v_{jk}]^\top[v_{ij},v_{jk}]\big)^{-1}[v_{ij},v_{jk}]^\top$ and 
$G_{ijk}[v_{ij},v_{jk}]=[e_t,0]$.
By (\ref{eq:inv}), (\ref{eq:A_ijk}), $f_{jk} = -f_{ki}-f_{ij}$, and $v_{jk} = -v_{ki}-v_{ij}$,
\begin{align*}
    \langle P_{ijk},G_{ijk}\rangle 
    &= \big\langle [f_{ij},f_{jk}]\big([v_{ij},v_{jk}]^\top[v_{ij},v_{jk}]\big)^{-1}, 
      [e_t,0]\big\rangle
    = \frac{\big\langle [f_{ij},f_{jk}][v_{jk},v_{ji}]^\top[v_{jk},v_{ji}], 
        [e_t,0]\big\rangle}
    {4A_{ijk}^2}\\
    &=\frac{e_t^\top (f_{ij} v_{jk}^\top v_{jk}+f_{jk}v_{ji}^\top v_{jk})}
    {4A_{ijk}^2}
    =\frac{e_t^\top (f_{ij}\cot\beta_{ij}-f_{ki} \cot\beta_{ki})}{2A_{ijk}}.
\end{align*}

Let $Q$ be an orthogonal basis $T_{ijk}$, and represent $v\in V_{ijk}$ as $v = Q[u_1,u_2]^\top$. By definition, 
\begin{align*}
    \Big\langle \big(\frac{\partial}{\partial u_1}, 
        -\frac{\partial}{\partial u_2} \big),g \Big\rangle
    &= \frac{\partial}{\partial u_1}e_1^\top\big(G_{ijk}Q[u_1,u_2]^\top+c_i\big)
        -\frac{\partial}{\partial u_2}e_2^\top\big(G_{ijk}Q[u_1,u_2]^\top+c_i\big)\\
    &= e_1^\top G_{ijk}Qe_1-e_2^\top G_{ijk}Q e_2.
\end{align*}
We choose $Q = [v_{ij},v_{jk}]\big([v_{ij},v_{jk}]^\top [v_{ij},v_{jk}]\big)^{-1/2}$.
Then, by (\ref{eq:inv}), 
\begin{align} \label{eq:-12}
    G_{ijk}Q = e_te_1^\top\big([v_{ij},v_{jk}]^\top [v_{ij},v_{jk}]\big)^{-1/2}
    = \frac{e_te_1^\top}{2A_{ijk}}\big([v_{jk},v_{ji}]^\top [v_{jk},v_{ji}]\big)^{1/2}
    = \frac{e_t}{2A_{ijk}}(a_i, b_i),
\end{align}
where  
$
    a_i =\frac{\|v_{ij}\|^2+2A_{ijk}}{\sqrt{\|v_{ij}\|^2+\|v_{jk}\|^2+4A_{ijk}}},\quad
    b_i = \frac{\|v_{ij}\|\|v_{jk}\|\cos\beta_{ki}}
        {\sqrt{\|v_{ij}\|^2+\|v_{jk}\|^2+4A_{ijk}}}.
$ 
Hence,
\begin{align*}
    \Big\langle \big(\frac{\partial}{\partial u}, 
        -\frac{\partial}{\partial v} \big),g \Big\rangle
    &= \left\{\begin{array}{cc}
         \frac{a_i}{2A_{ijk}}, & {\rm if}\ t=1; \\
         \frac{-b_i}{2A_{ijk}}, & {\rm if}\ t=2.
    \end{array}\right.
\end{align*}
Since $\int_{S_V}\big\langle\nabla_{S_V}f,\nabla_{S_V}g\big\rangle\text{d}\sigma
    = \sum_{V_{ijk}\in{\cal N}(i)}\langle P_{ijk},G_{ijk}\rangle A_{ijk}$, the equality (\ref{LB_weakVS}) becomes
\begin{align}\label{f:i}
    \sum_{V_{ijk}\in{\cal N}(i)}w_{ij}f_j
    =\sum_{V_{ijk}\in{\cal N}(i)}\Big\{f_{ij}\cot\beta_{ij}-f_{ki}\cot\beta_{ki}\Big\}
    = \left\{\begin{array}{cl}
         0, & p\notin V_{ijk}, \\
       \frac{(a_i, -b_i)^\top}{2A_{ijk}},
        & p\in V_{ijk}.
    \end{array}\right.
\end{align}
It yields a linear system $L{\bf f} = {\bf b}$, where $\mathbf{f}$ is an $n_V\times 2$ matrix of rows $\{f_i^\top\}$, ${\bf b}$ is an $n_V \times 2$ matrix with three nonzero rows $\frac{(a_i, -b_i)}{2A_{ijk}},\frac{(a_j, -b_j)}{2A_{ijk}},\frac{(a_k, -b_k)}{2A_{ijk}}$ for $V_{ijk}$ containing $p$, and $L$ is the Laplacian matrix of order $n_V$ as before.

\section{Discretization of Beltrami equation in the unit disk}\label{Appendix:BE D}

Given a complex function $\mu$ defined in the unit disk $D$ and $\|\mu\|_\infty<1$, the Beltrami equation is commonly given in the complex form
\begin{align}\label{BE:complex}
    \frac{\partial g}{\partial \bar z} = \mu\frac{\partial g}{\partial z},\quad 
    |z|\leq 1,
\end{align}
where $z = x+\sqrt{-1}y = x+\imath y$ is the complex form of the 2D point $\mathbf{x} = (x,y)^\top$, and $g = g_1+\imath g_2$ is also taken as a complex function defined in $D$, where both $g_1$ and $g_2$ are real. Without confusions, we also take $g$ as a real 2D mapping: $g(\mathbf{x}) = (g_1(\mathbf{x}),g_2(\mathbf{x}))^\top$ as we commonly used in this paper. 

The Beltrami equation can be rewritten in a real form. Practically, by definition of Wirtinger derivatives,
\begin{align*}
    \frac{\partial g}{\partial z}
    &=\frac{1}{2}\Big(\frac{\partial g}{\partial x}-
        \imath \frac{\partial g}{\partial y}\Big)
    = \frac{1}{2}\Big(\big(\frac{\partial g_1}{\partial x}
        +\frac{\partial g_2}{\partial y}\big)
        +\imath\big(\frac{\partial g_2}{\partial x}
        -\frac{\partial g_1}{\partial y}\big)\Big),\\
    \frac{\partial g}{\partial \bar z}
    &=\frac{1}{2}\Big(\frac{\partial g}{\partial x}
        +\imath \frac{\partial g}{\partial y}\Big)
    = \frac{1}{2}\Big(\big(\frac{\partial g_1}{\partial x}
        -\frac{\partial g_2}{\partial y}\big)
        +\imath\big(\frac{\partial g_2}{\partial x}
        +\frac{\partial g_1}{\partial y}\big)\Big).
\end{align*}
Hence, writing $\mu = \mu_1+\imath \mu_2$ with real functions $\mu_1$ and $\mu_2$, the Beltrami equation (\ref{BE:complex}) has the real form
\begin{align*}
    \frac{\partial g_1}{\partial x}
        -\frac{\partial g_2}{\partial y}
    &= \mu_1\big(\frac{\partial g_1}{\partial x}
        +\frac{\partial g_2}{\partial y}\big)
    -\mu_2\big(\frac{\partial g_2}{\partial x}
        -\frac{\partial g_1}{\partial y}\big),\\
    \frac{\partial g_2}{\partial x}
        +\frac{\partial g_1}{\partial y}
    &=\mu_1\big(\frac{\partial g_2}{\partial x}
        -\frac{\partial g_1}{\partial y}\big)
        +\mu_2\big(\frac{\partial g_1}{\partial x}
        +\frac{\partial g_2}{\partial y}\big).
\end{align*}
Equivalently, letting $\frac{\partial g}{\partial x}$ and $\frac{\partial g}{\partial y}$ be the partial derivatives of $g = (g_1,g_2)^\top$ with respective to $x$ and $y$,
\begin{align}\label{BE:2D}
    \Big((1-\mu_1)I+\mu_2J\Big)\frac{\partial g}{\partial x}
    =\Big(\mu_2I+(1+\mu_1)J\Big)\frac{\partial g}{\partial y},\quad
    I = \Big[{1\atop 0}\ \ {0\atop 1}\Big],\quad
    J = \Big[{0\atop -1}\ \ {1\atop 0}\Big].
\end{align}

Rewriting (\ref{BE:2D}) as 
\[
    \Big[\frac{\partial g}{\partial x},\frac{\partial g}{\partial y}\Big]
    \Big[{1-\mu_1\atop -\mu_2}\Big]
    =J
    \Big[\frac{\partial g}{\partial x},\frac{\partial g}{\partial y}\Big]
    \Big[{-\mu_2\atop 1+\mu_1}\Big]
    \quad{\rm or}\quad
    \Big[\frac{\partial g}{\partial x},\frac{\partial g}{\partial y}\Big]
    \Big[{-\mu_2\atop 1+\mu_1}\Big]
    =J
    \Big[\frac{\partial g}{\partial x},\frac{\partial g}{\partial y}\Big]
    \Big[{\mu_1-1\atop \mu_2}\Big],
\]
and coupling them together, we get
\[
    \Big[\frac{\partial g}{\partial x},\frac{\partial g}{\partial y}\Big]
    \Big[{1-\mu_1\atop -\mu_2}\  {-\mu_2\atop 1+\mu_1}\Big]
    = J
    \Big[\frac{\partial g}{\partial x},\frac{\partial g}{\partial y}\Big]
    \Big[{-\mu_2\atop 1+\mu_1}\ {\mu_1-1\atop \mu_2}\Big].
\]
Equivalently,
\begin{align}\label{BE:2D+}
    \Big[\frac{\partial g}{\partial x},\frac{\partial g}{\partial y}\Big]
    B
    =J
    \Big[\frac{\partial g}{\partial x},\frac{\partial g}{\partial y}\Big],
\end{align}
where 
\[
    B = \Big[{1-\mu_1\atop -\mu_2}\  {-\mu_2\atop 1+\mu_1}\Big]
    \Big[{-\mu_2\atop 1+\mu_1}\ {\mu_1-1\atop \mu_2}\Big]^{-1}
    = \begin{bmatrix}
     \frac{2\mu_2}{1-\mu_1^2-\mu_2^2}&  
     \frac{(1-\mu_1)^2+\mu_2^2}{1-\mu_1^2-\mu_2^2}\\
    \frac{-(1+\mu_1)^2-\mu_2^2}{1-\mu_1^2-\mu_2^2}& 
    \frac{-2\mu_2}{1-\mu_1^2-\mu_2^2}
    \end{bmatrix}.
\]

A discrete approach is to discretize $\mu(\mathbf{x})$ as $(\mu'_{ijk},\mu''_{ijk})$ in $V_{ijk}$ that yields a constant matrix $b_{ijk}$ of $B$ in $V_{ijk}$, and choose a continuous and piecewise linear function 
$g^h(\mathbf{x}) 
    = c_{ijk}+ G_{ijk}\mathbf{x}$ 
for $\mathbf{x} = (x,y)^\top\in V_{ijk}$,
where $c_{ijk}\in {\cal R}^2$ and $G_{ijk}\in{\cal R}^{2\times2}$ are constant.
Since $\Big[\frac{\partial g^h}{\partial x},\frac{\partial g^h}{\partial y}\Big] = G_{ijk}$ in $V_{ijk}$, and 
\begin{align}\label{G_ijk}
    G_{ijk}
    = [g_{ij},g_{jk}][v_{ij},v_{jk}]^{-1}
    = -\frac{1}{2A_{ijk}}(g_{ij}v_{jk}^\top-g_{jk}v_{ij}^\top)J,
\end{align}
the Beltrami equation (\ref{BE:2D+}) is then discretized as
\begin{align}\label{BE:2D disc}
    G_{ijk}b_{ijk}=JG_{ijk},\quad \forall V_{ijk}.
\end{align}

Since $G_{ijk}v_{jk} = g_{jk}$ by (\ref{G_ijk}), and $\sum_{V_{ijk}\in{\cal N}(i)}v_{jk} = 0$ for each inner point $v_i$. Hence, 
\begin{align*}
    0 &= J\sum_{V_{ijk}\in{\cal N}(i)}G_{ijk}v_{jk}
    = \sum_{V_{ijk}\in{\cal N}(i)}G_{ijk}b_{ijk}v_{jk}
    =\sum_{V_{ijk}\in{\cal N}(i)} \frac{1}{2A_{ijk}}(g_{ij}v_{jk}^\top-g_{jk}v_{ij}^\top)\hat v_{jk}\\
    &=\sum_{V_{ijk}\in{\cal N}(i)}\Big\{
    \frac{v_{jk}^\top\hat v_{jk}}{2A_{ijk}}g_i
    +\frac{v_{ki}^\top\hat v_{jk}}{2A_{ijk}}g_j
    +\frac{v_{ij}^\top\hat v_{jk}}{2A_{ijk}}g_k
    \Big\}, 
\end{align*}
where $\hat v_{jk} =-Jb_{ijk}v_{jk}$. This is a linear system $L\mathbf{g}= 0$, where $L\in{\cal R}^{n_{V_I}\times n_V}$ with the number $n_{V_I}$ of inner vertices and the number $n_V$ of all vertices,  $\mathbf{g}\in{\cal R}^{n_V\times2}$ has the rows $\{g_\ell^\top\}$. The entries $\{w_{ij}\}$ of $L$ under the indexing of vertices are
\begin{align*}
    w_{ij} = \left\{\begin{array}{ll}
    \frac{1}{2}\Big\{\frac{v_{ki}^\top\hat v_{jk}}{2A_{ijk}}
    +\frac{v_{ji}^\top\hat v_{k'j}}{2A_{ik'j}}\Big\},& V_{ijk}\in {\cal N}(i);\\
    0, & V_{ijk}\notin {\cal N}(i);
    \end{array}\right.\quad j\neq i,\quad
    w_{ii} = -\sum_{j}w_{ij}.
\end{align*}
That is the same as that given in \cite{LMKC13}. Notice that the boundary points $\mathbf{g}_B$ of $\mathbf{g}$ should be given. That is, it is an inhomogeneous linear system for inner points $\mathbf{g}_I$: $L_I\mathbf{g}_I = -L_B\mathbf{g}_B$.

\bibliographystyle{siamplain}
\bibliography{reference}
\end{document}

%% file: ex_shared.tex
\usepackage[utf8]{inputenc}
\usepackage{amsmath,amssymb,bm}
\usepackage{amsfonts}
\usepackage{mathrsfs}
\usepackage{graphicx}
\usepackage{epstopdf}
\usepackage{enumitem}
\usepackage{overpic}
\usepackage{multirow}
\usepackage{nicefrac}
\usepackage{subfig}
\usepackage{float}
\usepackage{caption}
\usepackage{multicol}
\usepackage{dsfont}
\usepackage{bbm}
\usepackage{bm}
\usepackage{url}
\usepackage{verbatim}
\usepackage{algorithm,algorithmic}
\usepackage{xcolor}
\usepackage{cleveref}
\usepackage{hyperref}
\hypersetup{
    colorlinks=true,
    linkcolor=blue,
    filecolor=blue,      
    urlcolor=blue,
    citecolor=blue,
}

\ifpdf
  \DeclareGraphicsExtensions{.eps,.pdf,.png,.jpg}
\else
  \DeclareGraphicsExtensions{.eps}
\fi

\newsiamremark{remark}{Remark}
\newsiamremark{hypothesis}{Hypothesis}
\crefname{hypothesis}{Hypothesis}{Hypotheses}
\newsiamthm{claim}{Claim}
\newsiamthm{question}{Question}

\headers{Convergence Analysis}{Zhenyue Zhang and Zhong-Heng Tan}

\title{Convergence Analysis of Discrete Conformal Transformation}

\author{
Zhenyue Zhang\thanks{Nanjing Center for Applied Mathematics, Nanjing, and Zhejiang University, Hangzhou, People’s Republic of China. (\email{zyzhang@zju.edu.cn}). Corresponding author.}
\and
Zhong-Heng Tan\thanks{School of Mathematics, Southeast University, and Nanjing Center for Applied Mathematics, Nanjing, People’s Republic of China. (\email{zhtan95@seu.edu.cn}).}
}
\date{}

\usepackage{amsopn}

